\documentclass[11pt, leqno, twosided]{article}
\usepackage{a4, indentfirst, latexsym, amssymb, amsmath, amsthm}
\usepackage{t1enc}

  \newtheorem{Th}{Theorem}[section]
  \newtheorem{Prop}[Th]{Proposition}
  \newtheorem{Lem}[Th]{Lemma}
  \newtheorem{Cor}[Th]{Corollary}
  
  \theoremstyle{definition}
  \newtheorem{Rem}[Th]{Remark}
  \theoremstyle{plain}

\begin{document}
\begin{center}
{\Large{\bf Periodic solutions for nonlinear evolution equations \\ at resonance}}
\\[10mm]
Piotr Kokocki \footnote{\parbox[t]{100mm}{Corresponding author. \\
{\em E-mail address:} p.kokocki@mat.umk.pl.} \\[2mm]
The researches supported by the MNISzW Grant no. N N201 395137 
} \\[2mm]

{\em Faculty of Mathematics and Computer Science\\
Nicolaus Copernicus University\\
ul. Chopina 12/18, 87-100 Toru\'n, Poland }\\[18mm]

\end{center}

\begin{abstract}
We are concerned with periodic problems for nonlinear evolution equations at resonance of the form $\dot u(t) = - A u(t) + F (t,u(t))$, where a densely defined linear operator $A\colon D(A)\to X$ on a Banach space $X$ is such that $-A$ generates a compact $C_0$ semigroup and $F\colon [0,+\infty)\times X \to X$ is a nonlinear perturbation. Imposing appropriate Landesman--Lazer type conditions on the nonlinear term $F$, we prove a formula expressing the fixed point index of the associated translation along trajectories operator, in the terms of a time averaging of $F$ restricted to $\mathrm{Ker} \, A$. By the formula, we show that the translation operator has a nonzero fixed point index and, in consequence, we conclude that the equation admits a periodic solution.
\end{abstract}

\section{Introduction}

Consider a periodic problem
\begin{equation}\label{A1}
\left\{\begin{array}{ll} \dot u(t) = - A u(t) + F (t,u(t)), & \qquad t > 0 \\
u(t) = u(t + T) & \qquad t \ge 0,
\end{array}\right.
\end{equation}
where $T > 0$ is a fixed period, $A\colon D(A)\to X$ is a linear operator such that $-A$ generates a $C_0$ semigroup of bounded linear operators on a Banach space $X$ and $F\colon [0,+\infty)\times X \to X$ is a continuous mapping. The periodic problems are the abstract formulations of many differential equations including the parabolic partial differential equations on an open set $\Omega\subset\mathbb{R}^n$, with smooth boundary, of the form
\begin{equation}
\left\{ \begin{array}{ll}
u_t = - \mathcal{A} u + f(t, x, u) & \qquad\text{in}\quad  (0,+\infty)\times\Omega \\
\mathcal{B} u = 0 & \qquad\text{on}\quad [0,+\infty)\times\partial\Omega \\
u(t,x) = u(t + T,x) & \qquad\text{in}\quad [0,+\infty)\times\Omega,
\end{array} \right.
\end{equation}
where $$\mathcal{A}u = - D_i(a^{ij} D_j u) + a^k D_k u + a_0 u$$ is such that $a^{ij} = a^{ji}\in C^1(\overline\Omega)$, $a^k, a_0\in C(\overline\Omega)$,
$$a^{ij}(x)\xi_i \xi_j \ge \theta |\xi|^2 \qquad\text{for}\quad \xi = (\xi_1,\xi_2, \ldots, \xi_n)\in \mathbb{R}^n, \quad x\in\Omega,$$ $\mathcal{B}$ stands for the Dirichlet or Neumann boundary operator and $f\colon [0,+\infty)\times\Omega\times\mathbb{R}\to \mathbb{R}$ is a continuous mapping. \\
\indent Given $x\in X$, let $u(t;x)$ be a (mild) solution of $$\dot u(t) = - A u(t) + F (t,u(t)), \qquad t > 0$$ such that $u(0;x) = x$. We look for the $T$-periodic solutions of (\ref{A1}) as the fixed points of the translation along trajectory operator $\Phi_T \colon X \to X$ given by $\Phi_T(x) := u(T;x)$. \\
\indent One of the effective methods used to prove the existence of the fixed points of $\Phi_T$ is
\emph{the averaging principle} involving the equations
\begin{equation}\label{A2}
\dot u(t) = - \lambda A u(t) + \lambda F (t,u(t)), \qquad t > 0
\end{equation}
where $\lambda > 0$ is a parameter. Let $\Theta^\lambda_T\colon X \to X$ be the translation operator for (\ref{A2}). It is clear that $\Phi_T = \Theta^1_T$. Define the mapping $\widehat{F}\colon X \to X$ by $\widehat{F}(x) := \frac{1}{T}\int_0^TF(s,x) \, d s$ for $x\in X$. The averaging principle says that for every open bounded set $U\subset X$ such that $0 \notin (-A + \widehat{F})(D(A)\cap \partial U)$, one has that $\Theta^\lambda_T(x)\neq x$ for $x\in\partial U$ and
$$\deg(I - \Theta^\lambda_T, U) = \deg(-A + \widehat{F}, U)$$ provided $\lambda > 0$ is sufficiently small. In the above formula $\deg$ stands for the appropriate topological degree. Therefore, if $\deg(-A + \widehat{F}, U)\neq 0$, then using suitable \emph{a priori} estimates and the continuation argument, we infer that $\Theta^1_T$ has a fixed point and, in consequence, (\ref{A1}) admits a periodic solution starting from $\overline U$. The averaging principle for periodic problems on finite dimensional manifolds was studied in \cite{Furi-Pera}. The principle for the equations on any Banach space has been recently considered in \cite{Cwiszewski-1} in the case when $-A$ generates a compact $C_0$ semigroup and in \cite{Cwiszewski-2} for $A$ being an \emph{m}-accretive operator. In \cite{Cwiszewski-Kokocki}, a similar results were obtained when $-A$ generates a semigroup of contractions and $F$ is condensing. For the results when the operator $A$ is replaced by a time-dependent family $\{A(t)\}_{t \ge 0}$ see \cite{Cwisz-Kok-2}. \\
However there are examples of equations where the averaging principle in the above form is not applicable. Therefore, in this paper, motivated by \cite{Brez-Nir}, \cite{Amb-Man}, \cite{hess} and \cite{Schi-Schm}, we use the method of translation along trajectories operator to derive its counterpart in the particular situation when the equation (\ref{A1}) is at resonance i.e., $\mathrm{Ker}\, A \neq 0$ and $F$ is bounded. Let $N:=\mathrm{Ker} \, A$ and assume that the $C_0$ semigroup $\{S_A(t)\}_{t\ge 0}$ generated by $-A$ is compact. Then it is well known that (real) eigenvalues of $S_A(T)$ make a sequence which is either finite or converges to $0$ and the algebraic multiplicity of each of them is finite. Denote by $\mu$ the sum of the algebraic multiplicities of eigenvalues of $S_A(T)\colon X\to X$ lying in $(1,+\infty)$. Since the semigroup is compact, the operator $A$ has compact resolvents and, in consequence, $\dim N < +\infty$. Let $M$ be a subspace of $X$ such that $N\oplus M = X$ with $S_A(t)M \subset M$ for $t\ge 0$. Define a mapping $g\colon N \to N$ by
\begin{equation}\label{ggg1}
g(x) := \int_0^T PF(s,x)\,d s \qquad\mathrm{for}\quad x\in N
\end{equation}
where $P\colon X \to X$ is a topological projection onto $N$ with $\mathrm{Ker}\, P = M$. \\
\indent First, we are concerned with an equation $$\dot u(t) = - A u(t) + \varepsilon F (t,u(t)), \qquad t > 0$$ where $\varepsilon\in [0,1]$ is a parameter. Denoting by $\Phi^\varepsilon_t\colon X \to X$ the translation along trajectory operator associated with this equation, we shall show that, if $V\subset M$ is an open bounded set, with $0\in V$ and $U\subset N$ is an open bounded set in $N$ such that $g(x) \neq 0$ for $x$ from the boundary $\partial_N U$ of $U$ in $N$, then for small $\varepsilon \in (0,1)$, $\Phi^\varepsilon_T(x) \neq x$ for $x\in\partial (U\oplus V)$ and
\begin{equation}\label{ddegg}
\mathrm{deg_{LS}}(I - \Phi_T^{\varepsilon}, U\oplus V) = (-1)^{\mu + \dim N}\mathrm{deg_B}(g, U).
\end{equation}
Here $\mathrm{deg_{LS}}$ and $\mathrm{deg_B}$ stand for the Leray--Schauder and Brouwer degree, respectively. The obtained result improves that from \cite{Schi-Schm}. \\
\indent Further, for an open and bounded set $\Omega\subset \mathbb{R}^n$, we shall use the formula (\ref{ddegg}) to study the periodic problem
\begin{equation}\label{A-k-F}
\left\{\begin{array}{ll} \dot u(t) = - A u(t) + \lambda u(t) + F (t,u(t)), & \qquad t > 0 \\
u(t) = u(t + T) & \qquad t \ge 0,
\end{array}\right.
\end{equation}
where $A\colon D(A) \to X$ is a linear operator on the Hilbert space $X:=L^2(\Omega)$ with a real eigenvalue $\lambda$ and $F\colon [0,+\infty) \times X \to X$ is a continuous mapping. As before we assume that $-A$ generates a compact $C_0$ semigroup $\{S_A(t)\}_{t\ge 0}$ on $X$. The mapping $F$ is associated with a bounded and continuous $f\colon [0,+\infty) \times \Omega \times \mathbb{R} \to \mathbb{R}$ as follows
\begin{equation}\label{row3}
F(t,u)(x) := f(t,x,u(x)) \qquad\mathrm{for}\quad t\in [0,+\infty),\quad x\in\Omega.
\end{equation}
Additionally we suppose that the following kernel coincidence holds true (which is more general than to assume that A is self-adjoint)
$$N_\lambda:= \mathrm{Ker}\,(A - \lambda I) = \mathrm{Ker}\,(A^* - \lambda I) = \mathrm{Ker}\,(I - e^{\lambda T}S_A(T)).$$
Let $\Psi_t\colon  X \to X$ be the translation along trajectories operator associated with the equation
$$\dot u(t) = - A u(t) + \lambda u(t) + F (t,u(t)), \qquad t > 0.$$
The formula (\ref{ddegg}), under suitable Landesman--Lazer type conditions introduced in \cite{Lad-Laz}, gives an effective criterion for the existence of $T$-periodic solutions of (\ref{A-k-F}). Namely, we prove that there is an $R>0$ such that $g(x) \neq 0$ for $x\in N_\lambda \setminus B(0,R)$, $\Psi_T(x)\neq x$ for $x\in X\setminus B(0,R)$ and
\begin{equation}\label{ind12}
\mathrm{deg_{LS}}(I - \Psi_T, B(0,R)) = (-1)^{\mu(\lambda) + \dim N_\lambda}\,\mathrm{deg_B}(g, B(0,R)\cap N_\lambda )
\end{equation}
where $\mu(\lambda)$ is the sum of the algebraic multiplicities of the eigenvalues of $e^{\lambda T} S_{A}(T)$ lying in $(1,+\infty)$ and $g\colon N_\lambda  \to N_\lambda $ is given by (\ref{ggg1}) with $P$ being the orthogonal projection on $N_\lambda $. Additionally, we compute $\mathrm{deg_B}(g, B(0,R)\cap N_\lambda )$, which may be important in the study of problems concerning to the multiplicity of periodic solutions. Obtained applications correspond to those from \cite{Brez-Nir}, \cite{hess}, where a different approach were used to prove the existence of periodic solutions for parabolic equations at resonance. For the results concerning hyperbolic equations see e.g. \cite{Cwiszewski-3}, \cite{Ces-Kann}, \cite{Fuc-Maw}

\emph{Notation and terminology.} Throughout the paper we use the following notational conveniences.
If $(X, \|\cdot\|)$ is a normed linear space, $Y\subset X$ is a subspace and $U\subset Y$ is a subset, then by $\mathrm{cl}_{\,Y} \, U$ and $\partial_{\,Y} \, U$ we denote the closure and boundary of $U$ in $Y$, respectively, while by $\mathrm{cl} \, U \ (\overline U)$  and $\partial \, U$ we denote the closure and boundary of $U$ in $X$, respectively. If $Z$ is a subspace of $X$ such that $X = Y\oplus Z$, then for subsets $U\subset Y$ and $V\subset Z$ we write $U\oplus V:=\{x + y \ | \ x\in U, \ y\in V\}$ for their algebraic sum. We recall also that a $C_0$ semigroup $\{S(t)\colon X \to X\}_{t\ge 0}$ is compact if $S(t) V$ is relatively compact for every bounded $V\subset X$ and
$t > 0$. \\

\section{Translation along trajectories operator}

Consider the following differential problem
\begin{equation}\label{A-F-lam}
\left\{\begin{array}{ll}
\dot u(t) = - A u(t) + F (\lambda,t,u(t)), & \quad t > 0 \\
u(0) = x
\end{array}\right.
\end{equation}
where $\lambda$ is a parameter from a metric space $\Lambda$, $A\colon D(A) \to X$ is a linear operator on a Banach space $(X, \|\cdot\|)$ and $F\colon \Lambda\times [0,+\infty)\times X \to X$ is a continuous mapping. In this section $X$ is assumed to be real, unless otherwise stated. Suppose that $-A$ generates a compact $C_0$ semigroup $\{S_A(t)\}_{t\ge 0}$ and the mapping $F$ is such that \\[2mm]
\makebox[10mm][r]{(F1)} \parbox[t]{138mm}{for any $\lambda \in \Lambda$ and $x_0\in X$ there is a neighborhood $V \subset X$ of $x_0$ and a constant $L > 0$ such that for any $x,y\in V$
$$\|F(\lambda,t,x) - F(\lambda, t, y)\|\le L\|x - y\| \qquad\mathrm{for}\quad t\in[0,+\infty);$$} \\[1mm]
\makebox[10mm][r]{(F2)} \parbox[t]{138mm}{there is a continuous function $c\colon [0,+\infty) \to [0,+\infty)$ such that
    \[\|F(\lambda, t, x)\| \le c(t)(1 + \|x\|) \qquad\mathrm{for}\quad \lambda\in \Lambda, \quad t\in [0,+\infty),\quad x\in X. \]} \\
\indent A mild solution of the problem (\ref{A-F-lam}) is, by definition, a continuous mapping $u\colon [0, +\infty) \to X$ such that
\begin{equation*}
u(t) = S_A(t)x + \int_{0}^tS_A(t - s)F(\lambda,s,u(s))\,d s \qquad\text{ for }\quad t\ge 0.
\end{equation*}
It is well known (see e.g. \cite{Pazy}) that for any $\lambda\in\Lambda$ and $x\in X$, there is unique mild solution $u(\,\cdot \, ; \lambda, x)\colon [0,+\infty) \to X$ of (\ref{A-F-lam}) such that $u(0; \lambda, x) = x$ and therefore, for any $t\ge 0$, one can define \emph{the translation along trajectories operator} $\Phi_t\colon \Lambda\times X \to X$ by \[\Phi_t(\lambda, x) := u(t \, ;\lambda, x) \qquad\text{for}\quad \lambda\in\Lambda, \quad x\in X.\]
As we need the continuity and compactness of $\Phi_t$, we recall the following
\begin{Th}\label{tw-exi-con-comp2}
Let $A\colon D(A) \to X$ be a linear operator such that $-A$ generates a compact $C_0$ semigroup and let $F\colon \Lambda\times [0,+\infty)\times X \to X$ be a continuous mapping such that conditions (F1) and (F2) hold. \\[2mm]
\makebox[10mm][r]{(a)} \parbox[t]{138mm}{If sequences $(\lambda_n)$ in $\Lambda$ and $(x_n)$ in $X$ are such that $\lambda_n \to \lambda_0$ and $x_n \to x_0$, as $n\to +\infty$, then
\begin{equation*}
u(t ; \lambda_n, x_n) \to u(t ; \lambda_0, x_0) \qquad\mathrm{as}\quad n\to +\infty,
\end{equation*}
uniformly for $t$ from bounded intervals in $[0,+\infty)$.}\\[1mm]
\makebox[10mm][r]{(b)} \parbox[t]{138mm}{For any $t > 0$, the operator $\Phi_t\colon \Lambda\times X \to X$ is completely continuous, i.e. $\Phi_t (\Lambda\times V)$ is relatively compact, for any bounded $V\subset X$. }
\end{Th}

\begin{Rem}
The above theorem is slightly different from Theorem 2.14 in \cite{Cwiszewski-1}, where it is proved in the case when linear operator is dependent on parameter as the mapping $F$, and moreover the parameter space $\Lambda$ is compact. The above theorem says that if $A$ is free of parameters, then compactness of $\Lambda$ may be omitted.
\end{Rem}
Before we start the proof we prove the following technical lemma
\begin{Lem}\label{rem-bounded4}
Let $\Omega\subset X$ be a bounded set. Then \\[1mm]
\makebox[10mm][r]{(a)} \parbox[t]{138mm}{for every $t_0 > 0$ the set $\left\{u(t \ ;\lambda,x) \ | \ t\in [0,t_0], \ \lambda\in\Lambda, \ x\in\Omega\right\}$ is bounded;}
\makebox[10mm][r]{(b)} \parbox[t]{138mm}{for every $t_0 > 0$ and $\varepsilon > 0$ there is $\delta > 0$ such that if $t,t'\in [0,t_0]$ and $0 < t' - t <\delta$, then
\[\left\|\int_t^{t'} S_A(t' - s)F(\lambda, s, u(s ; \lambda, x)) \,ds \right\| \le \varepsilon \qquad\text{for}\quad \lambda\in\Lambda, \quad x\in\Omega;\] }
\makebox[10mm][r]{(c)} \parbox[t]{138mm}{for every $t_0 > 0$ the set
\[S(t_0):=\left\{\int_0^{t_0} S_A(t_0 - s)F(\lambda, s, u(s ; \lambda, x)) \,ds \ \Big| \ \lambda\in\Lambda, \ x\in\Omega\right\}\] is bounded.}
\end{Lem}
\begin{proof}
Throughout the proof we assume that the constants $M\ge 1$ and $\omega\in\mathbb{R}$ are such that $\|S_A(t)\|\le Me^{\omega t}$ for $t\ge 0$.
(a) Let $R > 0$ be such that $\Omega\subset B(0,R)$. Then by condition (F2), for every $t\in[0,t_0]$
\begin{align*}
\|u(t;\lambda, x)\| & \le \|S_A(t) x\| + \int_0^t \|S_A(t - s)F(\lambda,s,u(s;\lambda, x))\|\,ds \\
& \le M e^{|\omega| t}\|x\| + \int_0^t M e^{|\omega| (t - s)}c(s)(1 + \|u(s;\lambda, x)\|)\,ds \\
& \le R M e^{|\omega| t_0} + t_0 K M e^{|\omega | t_0} + \int_0^t K M e^{|\omega | t_0}\|u(s;\lambda, x)\|\,ds,
\end{align*}
where $K := \sup_{s\in [0,t_0]}c(s)$. By the Gronwall inequality
\begin{equation}\label{ine-bound}
\|u(t ; \lambda, x)\| \le C_0 e^{t_0 C_1} \qquad\text{for}\quad t\in [0,t_0],\quad \lambda\in\Lambda \quad x\in\Omega,
\end{equation}
where $C_0 := R M e^{|\omega | t_0} + t_0 K M e^{|\omega | t_0}$ and
$C_1 := K M e^{|\omega | t_0}$. \\
(b) From (a) it follows that there is $C > 0$ such that
$\|u(t ; \lambda, x)\| \le C$ for $t\in [0,t_0]$, $\lambda\in\Lambda$ and $x\in\Omega$. Therefore, if $t,t'\in [0,t_0]$ are such that $t < t'$, then
\begin{multline*}
\left\|\int_t^{t'} S_A(t' - s)F(\lambda, s, u(s ; \lambda, x)) \,ds \right\|
\le \int_t^{t'} M e^{\omega (t' - s)}\|F(\lambda, s, u(s ; \lambda, x))\| \,ds \\
\le \int_t^{t'} M e^{|\omega |(t' - s)}c(s)(1 + \|u(s ; \lambda, x)\|) \,ds
= (t' - t) M K e^{|\omega | t_0}(1 +  C).
\end{multline*}
Taking $\delta:= \varepsilon(M K e^{|\omega | t_0}(1 + C))^{-1}$ we obtain the assertion. \\
(c) For any $\lambda\in\Lambda$ and $x\in\Omega$
\begin{multline*}
\left\|\int_0^{t_0} S_A(t_0 - s)F(\lambda, s, u(s ; \lambda, x)) \,ds \right\|
\le \int_0^{t_0} M e^{\omega (t_0 - s)}c(s)(1 + \|u(s ; \lambda, x)\|) \,ds \\
\le \int_0^{t_0} M K e^{|\omega | t_0}(1 +  \|u(s ; \lambda, x))\|) \,ds
\le t_0 M K e^{|\omega | t_0}(1 +  C)
\end{multline*}
and $S(t_0)$ is bounded as claimed.
\end{proof}

\begin{proof}[Proof of Theorem \ref{tw-exi-con-comp2}]Let $\Omega\subset X$ be a bounded set and let $t\in (0,+\infty)$. We shall prove first that the set $\Phi_t(\Lambda\times\Omega)$ is relatively compact. Let $\varepsilon > 0$. For $0 < t_0 < t$, $\lambda\in\Lambda$ and $x\in\Omega$
\begin{align*}
u(t ; \lambda, x) & = S_A(t)x + S_A(t - t_0)\left(\int_0^{t_0} S_A(t_0 - s)F(\lambda, s, u(s ; \lambda, x)) \,ds \right) \\
& \quad + \int_{t_0}^t S_A(t - s)F(\lambda, s, u(s ; \lambda, x)) \,ds,
\end{align*}
and, in consequence,
\begin{multline}\label{c-e-3a}
\{u(t ; \lambda, x) \ | \ \lambda\in\Lambda, \ x\in\Omega \}
\subset S_A(t)\Omega + S_A(t - t_0)D_{t_0} \\
\quad + \left\{\int_{t_0}^t S_A(t_0 - s)F(\lambda, s, u(s ; \lambda, x)) \,ds \ \Big| \ \lambda\in\Lambda, \ x\in\Omega\right\},
\end{multline}
where
\begin{equation*}
D_{t_0} := \left\{\int_0^{t_0} S_A(t_0 - s)F(\lambda, s, u(s ; \lambda, x)) \,ds \ \Big| \ \lambda\in\Lambda, \ x\in\Omega\right\}.
\end{equation*}
Applying Lemma \ref{rem-bounded4} (b), we infer that $t_0\in(0,t)$ may be chosen so that
\begin{equation}\label{c-e-1a}
\left\|\int_{t_0}^t S_A(t - s)F(\lambda, s, u(s ; \lambda, x)) \,ds \right\| \le \varepsilon \qquad\text{for}\quad \lambda\in\Lambda, \quad x\in\Omega.
\end{equation}
From the point (c) of this lemma it follows that $D_{t_0}$ is bounded. Combining (\ref{c-e-3a}) with (\ref{c-e-1a}) yields
\begin{align*}
\Phi_t(\Lambda\times\Omega) = \{u(t ; \lambda, x) \ | \ \lambda\in\Lambda, \ x\in\Omega \} \subset V_\varepsilon + B(0,\varepsilon)
\end{align*}
where $V_\varepsilon := S_A(t)\Omega + S_A(t - t_0)D_{t_0}$.
This implies that $V_\varepsilon$ is relatively compact, since $\{S_A(t)\}_{t\ge 0}$ is a compact semigroup and the sets $\Omega$, $D_{t_0}$ are bounded. On the other hand $\varepsilon > 0$ may be chosen arbitrary small and therefore the set $\Phi_t(\Lambda\times\Omega)$ is also relatively compact. \\
\indent Let $(\lambda_n)$ in $\Lambda$ and $(x_n)$ in $X$ be sequences such that $\lambda_n \to \lambda_0\in\Lambda$ and $x_n \to x_0\in X$. We prove that $u(t ; \lambda_n, x_n) \to u(t ; \lambda_0, x_0)$ as $n\to +\infty$ uniformly on $[0,t_0]$ where $t_0 > 0$ is arbitrary. For every $n\ge 1$ write $u_n := u(\cdot ; \lambda_n, x_n)$. We claim that $(u_n)$ is an equicontinuous sequence of functions. Indeed, take $t\in [0,+\infty)$ and let $\varepsilon > 0$. If $h > 0$ then, by the integral formula,
\begin{align}\label{ja9}
u_n(t + h) - u_n(t) & = S_A(h)u_n(t) - u_n(t) \\ \notag
& \quad + \int_t^{t + h}S_A(t + h - s)F(\lambda_n,s,u_n(s)) \,ds .
\end{align}
Note that for every $t\in [0,+\infty)$ the set $\{u_n(t) \ | \ n\ge 1\}$ is relatively compact as proved earlier. For $t=0$ it follows from the convergence of $(x_n)$, while for $t\in(0,+\infty)$ it is a consequence of the fact that the set $\Phi_t(\Lambda\times \{x_n \ | \ n\ge 1\})$ is relatively compact. From the continuity of semigroup there is $\delta_0 > 0$ such that
\begin{equation}\label{eqc-3}
\|S_A(h)u_n(t) - u_n(t)\|\le \varepsilon /2 \qquad\text{for}\quad h\in (0,\delta_0),\quad n\ge 1.
\end{equation}
By Lemma \ref{rem-bounded4} (b) there is $\delta \in (0,\delta_0)$ such that for $h\in (0,\delta)$ and $n\ge 1$
\begin{equation}\label{eqc-4}
\left\|\int_t^{t + h}S_A(t + h - s)F(\lambda_n,s,u_n(s)) \,ds \right\| \le \varepsilon /2.
\end{equation}
Combining (\ref{ja9}), (\ref{eqc-3}) and (\ref{eqc-4}), for $h\in (0,\delta)$ we infer that,
\begin{align*}
\|u_n(t + h) - u_n(t)\| & \le \|S_A(h)u_n(t) - u_n(t)\| \\ \nonumber
& \quad + \left\|\int_t^{t + h}S_A(t + h - s)F(\lambda_n,s,u_n(s)) \,ds \right\|
\le \varepsilon/2 + \varepsilon/2 = \varepsilon
\end{align*}
for every $n\ge 1$. We have thus proved that $(u_n)$ is right-equicontinuous on $[0,+\infty)$. It remains to show that $(u_n)$ is left-equicontinuous. To this end take $t\in(0,+\infty)$ and $\varepsilon > 0$. If $h$ and $\delta $ are such that $0< h < \delta < t$, then
\begin{align}\label{eqc-6}
\|u_n(t) - u_n(t - h)\| & \le \|u_n(t) - S_A(\delta)u_n(t - \delta)\| \\ \notag
 & \quad + \|S_A(\delta)u_n(t - \delta) - S_A(\delta - h)u_n(t - \delta)\|  \\ \notag
 & \quad + \|S_A(\delta - h)u_n(t - \delta)- u_n(t - h)\|,
\end{align}
and consequently, for any $n\ge 1$,
\begin{align}\label{eqc-10}
\|u_n(t) - u_n(t - h)\| &\le\left\|\int_{t - \delta}^t S_A(t - s)F(\lambda_n,s,u_n(s))\,ds\right\|\\ \notag
 & \quad + \|S_A(\delta)u_n(t - \delta) - S_A(\delta - h)u_n(t - \delta)\| \\ \notag
 & \quad + \left\|\int_{t - \delta}^{t - h}S_A(t - h - s)F(\lambda_n,s,u_n(s)) \,ds \right\|.
\end{align}
By Lemma \ref{rem-bounded4} (b) there is $\delta\in (0,t)$ such that for every $t_1, t_2\in [0,t]$ with $0 < t_1 - t_2 < \delta$, we have
\begin{equation}\label{eqc-7}
\left\|\int_{t_1}^{t_2}S_A(t_2 - s)F(\lambda_n,s,u_n(s)) \,ds \right\| \le \varepsilon /3 \qquad\text{for}\quad  n\ge 1.
\end{equation}
Using again the relative compactness of $\{u_n(t) \ | \ n\ge 1\}$ where $t\in [0,+\infty)$ we can choose  $\delta_1\in (0,\delta)$ such that for every $h\in (0, \delta_1)$ and $n\ge 1$
\begin{equation}\label{eqc-8}
\|S_A(\delta)u_n(t - \delta) - S_A(\delta - h)u_n(t - \delta)\| \le \varepsilon /3.
\end{equation}
Taking into account (\ref{eqc-10}), (\ref{eqc-7}), (\ref{eqc-8}), for $h\in(0,\delta_1)$
\begin{align*}
\|u_n(t) - u_n(t - h)\| & \le \left\|\int_{t - \delta}^t S_A(t - s)F(\lambda_n,s,u_n(s)) \,ds \right\| \\
& \quad + \|S_A(\delta)u_n(t - \delta) - S_A(\delta - h)u_n(t - \delta)\| \\
& \quad + \left\|\int_{t - \delta}^{t - h}S_A(t - h - s)F(\lambda_n,s,u_n(s)) \,ds \right\|
 \le \varepsilon,
\end{align*}
and finally the sequence $(u_n)$ is left-equicontinuous on $(0,+\infty)$. Hence $(u_n)$ is equicontinuous at every $t\in [0,+\infty)$ as claimed.\\
\indent For every $n\ge 1$ write $w_n := {u_n}_{|[0,t_0]}$. We shall prove that $w_n \to w_0$ in $C([0,t_0],X)$ where $w_0 = u(\cdot \ ;\lambda_0,x_0)_{|[0,t_0]}$. It is enough to show that every subsequence of $(w_n)$ contains a subsequence convergent to $w_0$. Let $(w_{n_k})$ be a subsequence of $(w_n)$. Since $(w_{n_k})$ is equicontinuous on $[0,t_0]$ and the set $\{w_{n_k}(s) \ | \ n\ge 1\} = \{u_{n_k}(s) \ | \ n\ge 1\}$ is relatively compact for any $s\in [0,t_0]$, by the Ascoli-Arzela Theorem, we infer that $(w_{n_k})$ has a subsequence $(w_{n_{k_l}})$ such that $w_{n_{k_l}} \to w$ in $C([0,t_0],X)$ as $l\to +\infty$. For every $l\ge 1$ define a mapping $\phi_l:[0,t_0] \to X$ by
$$\phi_l(s) := S_A(t - s)F(\lambda_{n_{k_l}},s,w_{n_{k_l}}(s)).$$
From the continuity of $\{S_A(t)\}_{t\ge 0}$ and $F$, we deduce that $\phi_l \to \phi$ in $C([0,t_0],X)$, where $\phi:[0,t_0] \to X$ is given by $\phi(s) = S_A(t - s)F(\lambda_0,s,w_0(s))$. It is clear that
\begin{align*}
w_{n_{k_l}}(t') = S_A(t')x_0 + \int_0^{t'} \phi_l(s)\,ds \qquad\text{for}\quad t'\in[0,t_0],
\end{align*}
and therefore, passing to the limit with $l\to \infty$, we infer that for $t'\in[0,t_0]$
\begin{align*}
w_0(t') = S_A(t')x_0 + \int_0^{t'} \phi(s)\,ds  = S_A(t')x_0 + \int_0^{t'} S_A(t' - s)F(\lambda_0,s,w_0(s))\,ds.
\end{align*}
By the uniqueness of mild solutions, $w_0(t) = w(t)$ for $t'\in[0,t_0]$ and we conclude that $w_{n_{k_l}} \to w_0 = u(\cdot \ ;\lambda_0,x_0)$ as $l\to \infty$ and finally that $w_n \to w_0$ in $C([0,t_0],X)$. This completes the proof of point (a).
\end{proof}

If linear operator $A:D(A) \to X$ is defined on a complex space $X$, then {\em the point spectrum} of $A$ is the set
$\sigma_p(A) := \{\lambda\in\mathbb{C} \ | \ \text{ there exists } \ z\in X\setminus\{0\} \ \text{ such that } \ \lambda z - A z = 0\}$.
For a linear operator $A$ defined on a real space $X$, we consider its complex point spectrum in the following way (see \cite{Am1} or \cite{Daner}). By the complexification of $X$ we mean a complex linear space $(X_\mathbb{C}, +, \,\cdot\,)$, where $X_\mathbb{C}:= X\times X$, with the operations of addition $+\colon X_\mathbb{C}\times X_\mathbb{C} \to \mathbb{C}$ and multiplication by complex scalars \ $\cdot\,\colon \mathbb{C}\times X_\mathbb{C} \to \mathbb{C}$ given by
\begin{align*}
& (x_1, y_1) + (x_2, y_2) := (x_1 + x_2, y_1 + y_2) && \hspace{-10mm}\qquad\mathrm{for}\quad (x_1, y_1), (x_2, y_2)\in X_\mathbb{C}, \quad \mathrm{and}\\
& (\alpha + \beta i)\cdot (x,y) := (\alpha x - \beta y, \alpha y + \beta x)
&& \hspace{-10mm}\qquad\mathrm{for}\quad \alpha + \beta i \in \mathbb{C},\quad (x,y)\in X_\mathbb{C},
\end{align*}
respectively. For convenience, denote the elements $(x,y)$ of $X_\mathbb{C}$ by $x + yi$. If $X$ is a space with a norm $\|\cdot\|$, then the mapping $\|\cdot\|_\mathbb{C}\colon X_\mathbb{C} \to \mathbb{R}$ given by
\begin{equation*}
\|x + y i\|_\mathbb{C} := \sup_{\theta\in [-\pi, \pi]}\|\sin\theta x + \cos\theta y\|
\end{equation*}
is a norm on $X_\mathbb{C}$, and $(X_\mathbb{C}, \|\cdot\|_\mathbb{C})$ is a Banach space, provided $X$ is so. \emph{The complexification} of $A$ is a linear operator $A_\mathbb{C}\colon D(A_\mathbb{C}) \to X_\mathbb{C}$ given by
$$D(A_\mathbb{C}) := D(A)\times D(A) \quad\text{and}\quad A_\mathbb{C}(x + yi) := Ax + Ayi \qquad\mathrm{for}\quad x + yi \in D(A_\mathbb{C}).$$
Now, one can define the {\em complex point spectrum} of $A$ by $\sigma_p(A):=\sigma_p(A_\mathbb{C})$.
\begin{Rem}
If $-A$ is a generator of a $C_0$ semigroup $\{S_{A}(t)\}_{t\ge 0}$, then it is easy to check that the family $\{S_{A}(t)_\mathbb{C}\}_{t\ge 0}$ of the complexified operators is a $C_0$ semigroup of bounded linear operators on $X_\mathbb{C}$ with the generator $- A_\mathbb{C}$.
\end{Rem}

In the following proposition we mention some spectral properties of $C_0$ semigroups

\begin{Prop}(see \cite[Theorem 16.7.2]{H-F}\label{th-sem-spec})
If $- A$ is the generator of a $C_0$ semigroup $\{S_A(t)\}_{t\ge 0}$ of bounded linear operators on a complex Banach space $X$, then
\begin{equation*}
\sigma_p(S_A(t)) = e^{-t\sigma_p(A)} \setminus \{0\} \qquad\mathrm{for}\quad t > 0.
\end{equation*}
Furthermore, if $\lambda\in\sigma_p(A)$ then for every $t > 0$
\begin{equation}\label{th-sem-1}
\mathrm{Ker}\,(e^{- \lambda t} I - S_A(t)) = \overline{\mathrm{span}}\left(\bigcup_{k\in\mathbb{Z}}\mathrm{Ker}\,(\lambda_{k,t} I - A)\right)
\end{equation}
where $\lambda_{k,t} := \lambda + (2k\pi /t )i$ for $k\in\mathbb{Z}$.
\end{Prop}

\section{Averaging principle for equations at resonance}
In this section we are interested in the periodic problems of the form
\begin{equation}\label{A-F-lam1}
\left\{ \begin{array}{ll}
\dot u(t) = - A u(t) + \varepsilon F (t,u(t)), & \quad t > 0 \\
u(t) = u(t + T) &  \quad t\ge 0 \end{array} \right.
\end{equation}
where $T > 0$ is a fixed period, $\varepsilon \in [0,1]$ is a parameter, $A\colon D(A) \to X$ is a linear operator on a real Banach space $X$ and $F\colon [0,+\infty)\times X \to X$ is a continuous mapping. Suppose that $F$ satisfies (F1) and (F2) and $-A$ generates a compact $C_0$ semigroup $\{S_A(t)\}_{t\ge 0}$ such that \\[1mm]
\makebox[10mm][r]{(A1)} \parbox[t]{138mm}{$\mathrm{Ker}\, A = \mathrm{Ker}\, (I - S_A(T))\neq \{0\}$;} \\[1mm]
\makebox[10mm][r]{(A2)} \parbox[t]{138mm}{there is a closed subspace $M\subset X$, $M\neq \{0\}$ such that $X = \mathrm{Ker}\, A \oplus M$ and $S_A(t)M \subset M$ for $t\ge 0$.} \\[1mm]

\begin{Rem}\label{ker-a}
(a) If $A$ is any linear operator such that $-A$ generates a $C_0$ semigroup $\{S_A(t)\}_{t\ge 0}$, then it is immediate that $\mathrm{Ker}\, A\subset \mathrm{Ker}\, (I - S_A(t))$ for $t \ge 0$. \\[1mm]
(b) Condition (A1) can be characterized in terms of the point spectrum. Namely, (A1) is satisfied if and only if
\begin{equation}\label{rrr2}
\{(2k\pi/T) i \ | \ k\in\mathbb{Z}, \ k\neq 0\}\cap\sigma_p(A)= \emptyset.
\end{equation}
To see this suppose first that (A1) holds. If $(2k\pi /T) i\in \sigma_p(A)$ for some $k \neq 0$, then there is $z = x + yi \in X_\mathbb{C} \setminus \{0\}$ such that
\begin{equation}\label{rrr1}
A_\mathbb{C} z = (2k\pi /T) z i.
\end{equation}
We actually know that $-A_\mathbb{C}$ is a generator of the $C_0$ semigroup $\{S_{A_\mathbb{C}}(t)\}_{t\ge 0}$ with $ S_{A_\mathbb{C}}(t) = S_A(t)_\mathbb{C}$ for $t\ge 0$. Therefore, by Proposition \ref{th-sem-spec}, we find that $z\in \mathrm{Ker}\,(I - S_{A_\mathbb{C}}(T))$ and, in consequence, \[S_A(T)x + S_A(T)y i = x + y i.\]
By (A1), we get $A x = Ay = 0$ and finally $A_\mathbb{C} z = 0$, contrary to (\ref{rrr1}). Conversely, suppose that (\ref{rrr2}) is satisfied. Operator $A_\mathbb{C}$ as a generator of a $C_0$ semigroup is closed, and hence $\mathrm{Ker}\, A_\mathbb{C}$ is a closed subspace of $X_\mathbb{C}$. On the other hand, by (\ref{th-sem-1}) and (\ref{rrr2}),
\[\mathrm{Ker}\, (I - S_A(T)_\mathbb{C}) = \mathrm{Ker}\, (I - S_{A_\mathbb{C}}(T)) = \mathrm{cl} \,\mathrm{Ker}\, A_\mathbb{C} = \mathrm{Ker}\, A_\mathbb{C},\]
which implies that $\mathrm{Ker}\, (I - S_A(T)) = \mathrm{Ker}\, A$, i.e. (A1) is satisfied.
\end{Rem}
Since $X$ is a Banach space and $M$, $N$ are closed subspaces, there are projections $P\colon X \to X$ and $Q\colon X \to X$ such that $P^2 = P$, $Q^2=Q$, $P + Q = I$ and $\mathrm{Im}\, P = N$, $\mathrm{Im}\, Q = M$. Let $\Phi^\varepsilon_T\colon X \to X$ be the translation along trajectories operator associated with
$$\dot u(t) = - A u(t) + \varepsilon F (t,u(t)), \quad t > 0$$
and let $\mu$ denote the sum of the algebraic multiplicities of eigenvalues of $S_A(T)$ lying in $(1, +\infty)$. The compactness of the semigroup $\{S_A(t)\}_{t\ge 0}$, implies that the non-zero real eigenvalues of $S_A(T)$ \ form a sequence which is either finite or converges to $0$ and the algebraic multiplicity of each of them is finite. In both cases, only a finite number of eigenvalues is greater than $1$ and hence $\mu$ is well defined.

We are ready to formulate the main result of this section

\begin{Th}\label{th-aver-ker}
Let $g\colon N \to N$ be a mapping given by
\begin{equation*}
g(x) := \int_0^T PF(s,x)\,d s \qquad\mathrm{for}\quad x\in N
\end{equation*}
and let $U \subset N$ and $V \subset M$ with $0\in V$, be open bounded sets. If $g(x)\neq 0$ for
$x \in \partial_N U$, then there is $\varepsilon_0 \in (0,1)$ such that for any $\varepsilon\in (0,\varepsilon_0]$ and $x\in \partial(U \oplus V)$, $\Phi_T^{\varepsilon}(x)\neq x$ and
\begin{equation*}
\mathrm{deg_{LS}}(I - \Phi_T^{\varepsilon}, U \oplus V) = (-1)^{\mu + \dim N}\,\mathrm{deg_B}(g,U)
\end{equation*}
where $\mathrm{deg_{LS}}$ and $\mathrm{deg_B}$ stand for the Leray--Schauder and the Brouwer topological degree, respectively.
\end{Th}
\begin{proof}
Throughout the proof, we write $W := U\oplus V$ and $\Lambda:= [0,1]\times [0,1]\times \overline{W}$. For any $(\varepsilon, s, y)\in\Lambda$ consider the differential equation
\begin{equation}\label{A-G-lam}
\dot u(t) = - A u(t) + G (\varepsilon, s,y,t,u(t)), \qquad t > 0
\end{equation}
where $G\colon \Lambda \times [0,+\infty)\times X \to X$ is defined by
\begin{equation*}
G(\varepsilon,s,y,t,x) := \varepsilon PF(t,s x + (1 - s)Py) + \varepsilon s QF(t,x).
\end{equation*}
We check that $G$ satisfies condition (F1). Indeed, fix $(\varepsilon,s,y)\in \Lambda$ and take $x_0\in X$. If $s = 0$ then $G(\varepsilon,s,y,t,\,\cdot\,)$ is constant, hence we may suppose that $s\neq 0$. There are constants $L_0,L_1 > 0$ and neighborhoods $V_0,V_1\subset X$ of points $s x_0 + (1 - s)Py$ and $x_0$, respectively, such that
\[\|F(t,x_1) - F(t,x_2)\| \le L_0 \|x_1 - x_2\| \qquad\mathrm{for}\quad x_1,x_2\in V_0,\quad t\in [0,+\infty)\]
and
\[\|F(t,x_1) - F(t,x_2)\| \le L_1 \|x_1 - x_2\| \qquad\mathrm{for}\quad x_1,x_2\in V_1,\quad t\in [0,+\infty).\]
Then $V' := \frac{1}{s}(V_0 - (1 - s)Py) \cap V_1$ is open, $x_0\in V'$ and, for any $x_1,x_2\in V'$,
\begin{multline*}
\|G(\varepsilon,s,y,t,x_1) - G(\varepsilon,s,y,t,x_2)\|   \le \\
 \varepsilon \|P\|\|F(t,s x_1 + (1 - s)P y) - F(t,s x_2 + (1 - s)P y)\|
   + s\varepsilon \|Q\|\|F(t,x_1) - F(t,x_2)\| \\
  \le \varepsilon L_0\|P\|\|x_1 - x_2\| + s\varepsilon L_1\|Q\|\|x_1 - x_2\|
  \le (L_0\|P\|+ L_1\|Q\|)\|x_1 - x_2\|,
\end{multline*}
i.e. (F1) is satisfied. An easy computation shows that condition (F2) also holds true. \\ If $(\varepsilon,s,y)\in \Lambda$ and $x\in X$, then by $u(\, \cdot \, ;\varepsilon,s,y,x)\colon [0,+\infty) \to X$ we denote unique mild solution of (\ref{A-G-lam}) starting at $x$. For $t\ge 0$, let $\Theta_t\colon \Lambda\times X \to X$ be the translation along trajectories operator given by $$\Theta_t(\varepsilon,s,y,x) := u(t ; \varepsilon,s,y,x) \qquad\mathrm{for}\quad (\varepsilon,s,y)\in\Lambda,\quad x\in X,\quad t\in[0,+\infty).$$
For every $\varepsilon\in (0,1)$ we define the mapping $M^\varepsilon\colon [0,1]\times \overline{W}\to X$ by
\begin{equation*}
M^\varepsilon(s,x) := \Theta_T(\varepsilon,s,x,x).
\end{equation*}
Clearly $M^\varepsilon$ is completely continuous for every $\varepsilon \in (0,1)$. Indeed, by Theorem \ref{tw-exi-con-comp2} the operator $\Theta_T$ is completely continuous and, consequently, the set $\Theta_T(\Lambda\times\overline{W}) \subset X$ is relatively compact. Since
\begin{equation*}
M^\varepsilon([0,1]\times \overline{W})
= \Theta_T(\{\varepsilon\}\times [0,1]\times \overline{W}\times\overline{W})
\subset \Theta_T(\Lambda\times\overline{W}),
\end{equation*}
the set $M^\varepsilon([0,1]\times \overline{W})$ is relatively compact as well. \\
\indent Now we claim that there is $\varepsilon_0 \in (0,1)$ such that
\begin{equation}\label{rownow}
M^\varepsilon(s,x) \neq x \qquad\mathrm{for}\quad x\in\partial W,\quad s\in [0,1],\quad \varepsilon\in (0, \varepsilon_0].
\end{equation}
Suppose to the contrary that there are sequences $(\varepsilon_n)$ in $(0,1)$, $(s_n)$ in $[0,1]$ and $(x_n)$ in $\partial W$ such that $\varepsilon_n\to 0$ and
\begin{equation}\label{ciag-theta1}
\Theta_T(\varepsilon_n,s_n,x_n,x_n) = M^{\varepsilon_n}(s_n,x_n) = x_n \qquad\mathrm{for}\quad n\ge 1.
\end{equation}
We may assume that $s_n\to s_0$ with $s_0\in[0,1]$. By (\ref{ciag-theta1}) and the boundedness of $(x_n)\subset\partial W$, the complete continuity of $\Theta_T$ implies that $(x_n)$ has a convergent subsequence. Without loss of generality we may assume that $x_n \to x_0$ as $n\to +\infty$, for some $x_0\in \partial W$. After passing to the limit in (\ref{ciag-theta1}), by Theorem \ref{tw-exi-con-comp2} (a), it follows that
\begin{equation}\label{ciag-theta2}
\Theta_T(0,s_0, x_0,x_0) = x_0.
\end{equation}
On the other hand
\begin{equation}\label{equ-theta1}
\Theta_t(0,s_0, x_0,x_0) = S_A(t)x_0 \qquad\mathrm{for}\quad t \ge 0,
\end{equation}
which together with (\ref{ciag-theta2}) implies that $x_0 = S_A(T)x_0$. Condition (A1) yields $x_0\in\mathrm{Ker}\, A = N$ and hence $Qx_0 = 0$. Since $0\in V$, and the equality
\begin{equation*}
\partial (U\oplus V) = \partial_N U \oplus \mathrm{cl}_M V \cup \mathrm{cl}_N U\oplus \partial_M V
\end{equation*}
holds true, we infer that $x_0\in \partial_N U$. By using of Remark \ref{ker-a} (a) and (\ref{equ-theta1}) we also find that
\begin{equation}\label{eq-theta2}
\Theta_t(0,s_0,x_0,x_0) = S_A(t)x_0 = x_0 \qquad\mathrm{for}\quad t\ge 0.
\end{equation}
For every $n\ge 1$, write $u_n := u(\,\cdot \, ;\varepsilon_n,s_n,x_n,x_n)$ for brevity. As a consequence of (\ref{ciag-theta1})
\begin{align}\label{ja17}
x_n & = S_A(T)x_n + \varepsilon_n\int_0^T S_A(T - \tau) P F(\tau,s_n u_n(\tau)
+ (1 - s_n)Px_n)\, d \tau \\ \nonumber
& \qquad + \varepsilon_n s_n\int_0^T S_A(T - \tau) Q F(\tau,u_n(\tau))\, d \tau \qquad\mathrm{for}\quad n\ge 1.
\end{align}
The fact that the spaces $M, N\subset X$ are closed and $S_A(t)N\subset N$, $S_A(t)M\subset M$, for
$t\ge 0$, leads to
\begin{align}\label{ja17a}
\varepsilon_n\int_0^T S_A(T - \tau) P F(\tau,s_n u_n(\tau) + (1 - s_n)Px_n)\, d \tau \in N \qquad\text{ and } \\ \nonumber \varepsilon_n s_n\int_0^T S_A(T - \tau) Q F(\tau,u_n(\tau))\, d \tau \in M \qquad\mathrm{for}\quad n\ge 1.
\end{align}
Combining (\ref{ja17}) with (\ref{ja17a}) gives
\begin{equation*}
P x_n = S_A(T)P x_n + \varepsilon_n\int_0^T S_A(T - \tau) P F(\tau,s_n u_n(\tau)+ (1 - s_n)P x_n)\, d \tau \qquad\mathrm{for}\quad n\ge 1,
\end{equation*}
and therefore
\begin{equation}\label{eq-g1}
\int_0^T P F(\tau,s_n u_n(\tau)+ (1 - s_n)P x_n)\, d \tau = 0 \qquad\mathrm{for}\quad n\ge 1,
\end{equation}
since $P x_n \in \mathrm{Ker}\, A = \mathrm{Ker}\,(I - S_A(T))$ for $n\ge 1$.
By Theorem \ref{tw-exi-con-comp2} (a) and (\ref{eq-theta2}) the sequence $(u_n)$ converges uniformly on $[0,T]$ to the constant mapping equal to $x_0$, hence, passing to the limit in (\ref{eq-g1}), we infer that
\begin{equation*}
g(x_0) = \int_0^T P F(\tau, x_0)\, d \tau = 0.
\end{equation*}
This contradicts the assumption, since $x_0\in\partial_N U$, and proves (\ref{rownow}). \\
\indent By the homotopy invariance of topological degree we have
\begin{equation}\label{deg1}
\mathrm{deg_{LS}}(I - \Phi_T^\varepsilon, W) = \mathrm{deg_{LS}}(I - M^\varepsilon(1,\,\cdot\,), W) = \mathrm{deg_{LS}}(I - M^\varepsilon(0,\,\cdot\,), W)
\end{equation}
for $\varepsilon\in (0,\varepsilon_0]$. \\
\indent Let the mappings $\widetilde{M}_1^\varepsilon\colon\overline{U}\to N$ and $\widetilde{M}_2^\varepsilon\colon\overline{V}\to M$ be given by
\begin{align*}
& \widetilde{M}_1^\varepsilon(x_1) := x_1 + \varepsilon \int_0^T PF(s,x_1)\,d s && \hspace{-20mm}\text{for}\quad x_1\in\overline U, \\
& \widetilde{M}_2^\varepsilon(x_2) := S_A(T)_{|M}x_2 && \hspace{-20mm}\text{for}\quad x_2\in\overline V
\end{align*}
and let $\widetilde{M}^\varepsilon\colon \overline{U} \times \overline{V} \to N\times M$ be their product
\begin{equation*}
\widetilde{M}^\varepsilon(x_1,x_2) := (\widetilde{M}_1^\varepsilon(x_1),\widetilde{M}_2^\varepsilon(x_2)) \qquad\mathrm{for}\quad (x_1,x_2)\in \overline{U} \times \overline{V}.
\end{equation*}
For $\varepsilon\in (0,1)$ and $x\in X$
\begin{align*}
M^\varepsilon(0,x)  = S_A(T)x + \varepsilon\int_0^T S_A(T - \tau)PF(\tau,Px)\, d \tau
 = S_A(T)x + \varepsilon\int_0^T PF(\tau,Px)\, d \tau.
\end{align*}
and therefore it is easily seen that the mappings $M^\varepsilon(0,\,\cdot\,)$ and $\widetilde{M}^\varepsilon$ are topologically conjugate.
By the compactness of the $C_0$ semigroup $\{S_A(t):M \to M\}_{t\ge 0}$ and the fact that $\mathrm{Ker}\,(I - S_A(T)_{|M}) = 0$, we infer that the mapping
\[I - \widetilde{M}^\varepsilon_2\colon M \to M\]
is a linear isomorphism. By use of the multiplication property of topological degree, for any $\varepsilon \in (0,1)$,
\begin{align*}
\mathrm{deg_{LS}}(I - M^\varepsilon(0,\,\cdot \,), W) & = \mathrm{deg_{LS}}(I - \widetilde{M}^\varepsilon, U\times V) \\
& = \mathrm{deg_B}(I - \widetilde{M}_1^\varepsilon, U) \cdot \mathrm{deg_{LS}}(I - \widetilde{M}_2^\varepsilon, V).
\end{align*}
Combining this with (\ref{deg1}), we conclude  that
\begin{align*}
\mathrm{deg_{LS}}(I - \Phi_T^\varepsilon, W) & = \mathrm{deg_B}(-\varepsilon \, g, U) \cdot \mathrm{deg_{LS}}(I - S_A(T)_{|M}, V) \\ \nonumber
& = (-1)^{\dim N} \mathrm{deg_B}(g, U) \cdot \mathrm{deg_{LS}}(I - S_A(T)_{|M}, V),
\end{align*}
for $\varepsilon\in(0,\varepsilon_0]$. If $\lambda\neq 1$ and $k \ge 1$ is
an integer then, by (A1) and (A2), $$\mathrm{Ker}\,(\lambda I -
S_A(T))^k_{|M} = \mathrm{Ker}\,(\lambda I - S_A(T))^k.$$ Hence
$\sigma_p(S_A(T)_{|M}) = \sigma_p(S_A(T)) \setminus \{1\}$ and the
algebraic multiplicities of the corresponding eigenvalues are the
same. Therefore, by the standard spectral properties of compact
operators (see e.g. \cite[Theorem 12.8.3]{Dugundji-Granas}),
\begin{equation*}
\mathrm{deg_{LS}}(I - S_A(T)_{|M}, V) = (-1)^\mu,
\end{equation*}
and finally
\begin{equation*}
\mathrm{deg_{LS}}(I - \Phi_T^\varepsilon, W) = (-1)^{\mu + \dim N}\mathrm{deg_B}(g, U),
\end{equation*}
for every $\varepsilon\in (0,\varepsilon_0]$, which completes the proof.
\end{proof}

An immediate consequence of Theorem \ref{th-aver-ker} is the following

\begin{Cor}\label{cor-aver}
Let $U \subset N$ and $V \subset M$ with $0\in V$, be open bounded sets such that $g(x)\neq 0$ for
$x \in \partial_N U$. If $\mathrm{deg_B}(g, U) \neq 0$, then there is $\varepsilon_0 \in (0,1)$ such that for any $\varepsilon\in (0,\varepsilon_0]$ problem (\ref{A-F-lam1}) admits a $T$-periodic mild solution.
\end{Cor}

\section{Periodic problems with the Landesman--Lazer type conditions}

Let $\Omega\subset\mathbb{R}^n$, $n\ge 1$, be an open bounded set and let $X := L^2(\Omega)$. By $\|\cdot\|$ and $\langle\,\cdot \, , \,\cdot\,\rangle$ we denote the usual norm and scalar product on $X$, respectively. Assume that continuous mapping $f\colon [0,+\infty)\times\Omega\times \mathbb{R}\to \mathbb{R}$ satisfies the following conditions \\[2mm]
\makebox[10mm][r]{(a)} \parbox[t]{138mm}{there is a constant $m > 0$ such that\\[-3mm]
    \begin{equation*}
    |f(t,x,y)| \le m \qquad\mathrm{for}\quad t\in [0,+\infty),\quad x\in\Omega, \quad y\in \mathbb{R};
    \end{equation*}} \\
\makebox[10mm][r]{(b)} \parbox[t]{138mm}{there is a constant $L > 0$ such that for any $t\in [0,+\infty)$, $x\in\Omega$ and $y_1,y_2\in\mathbb{R}$ \\[-3mm]
    \begin{equation*}
    |f(t,x,y_1) - f(t,x,y_2)| \le L|y_1 - y_2|;
    \end{equation*}} \\
\makebox[10mm][r]{(c)} \parbox[t]{138mm}{$f(t, x, y) = f(t+T, x, y)$ \ for $t\in [0,+\infty)$, $x\in\Omega$ and $y\in\mathbb{R}$;} \\[2mm]
\makebox[10mm][r]{(d)} \parbox[t]{138mm}{there are continuous functions $f_+,f_-\colon [0,+\infty)\times\Omega \to \mathbb{R}$ such that
\begin{align*}
f_+(t,x)  = \lim_{y \to +\infty} f(t,x,y) \qquad \text{and} \qquad f_-(t,x)  = \lim_{y \to -\infty} f(t,x,y)
\end{align*}
for $t\in [0,+\infty)$ and $x\in\Omega$. }\\[5mm]
\indent Consider the following periodic differential problem
\begin{equation}\label{A-eps-res}
\left\{ \begin{array}{ll}
\dot u(t) = - A u(t) + \lambda u(t) +  F(t,u(t)), & \qquad t > 0\\
u(t) = u(t + T) & \qquad t\ge 0 \end{array} \right.
\end{equation}
where $A\colon D(A) \to X$ is a linear operator such that $-A$ generates a compact $C_0$ semigroup $\{S_A(t)\}_{t\ge 0}$ of bounded linear operators on $X$, $\lambda$ is a real eigenvalue of $A$ and $F\colon [0,+\infty)\times X \to X$ is a continuous mapping given by the formula
\begin{equation*}
F(t,u)(x) := f(t,x,u(x)) \qquad\mathrm{for}\quad t\in [0,+\infty),\quad x\in\Omega.
\end{equation*}
Additionally, we suppose that  \\[2mm]
\makebox[10mm][r]{(A3)} \parbox[t]{138mm}{ $\mathrm{Ker}\,(A - \lambda I) = \mathrm{Ker}\,(A^* - \lambda I) = \mathrm{Ker}\,(I - e^{\lambda T}S_{A}(T))$.} \\[2mm]
Recall that by assumptions (a) and (b), the mapping $F$ is well defined, bounded, continuous and Lipschitz uniformly with respect to time. Therefore, the translations along trajectories operator $\Psi_t\colon X \to X$ associated with
$$\dot u(t) = - A u(t) + \lambda u(t) +  F(t,u(t)), \qquad t > 0$$
is well-defined and completely continuous for $t > 0$, as a consequence of Theorem \ref{tw-exi-con-comp2}. Let $N_\lambda  := \mathrm{Ker}\,(\lambda I - A)$ and define
$g\colon N_\lambda  \to N_\lambda $ by
\begin{equation*}
g(u) := \int_0^T PF(t,u) \,d t \qquad\text{for}\quad u\in N_\lambda ,
\end{equation*}
where $P\colon X\to X$ is the orthogonal projection onto $N_\lambda $. Since $\{S_A(t)\}_{t\ge 0}$ is compact, $A$ has compact resolvents and $\dim N_\lambda  < \infty$. Furthermore note that, for any $u,z\in N_\lambda $,
\begin{align}\label{gr3}
\langle g(u),z \rangle & = \int_0^T \langle PF(t,u), z\rangle \,d t = \int_0^T \langle F(t,u), z \rangle \,d t \\ \nonumber
& = \int_0^T \!\!\! \int_\Omega  f(t,x,u(x))z(x) \,d x dt.
\end{align}
We are ready to state the main result of this section
\begin{Th}\label{th-guid-fun1}
Suppose that $f \colon  [0,+\infty)\times\Omega\times\mathbb{R} \to \mathbb{R}$ satisfies one of the following Landesman--Lazer type conditions:
\begin{equation}\label{lazer1}
\int_0^T \!\!\!\int_{\{u>0\}} f_+(t,x)u(x) \,d x dt +\int_0^T\!\!\!\int_{\{u<0\}} f_-(t,x) u(x)\,d x dt > 0,
\end{equation}
for any $u\in N_\lambda $ with $\|u\| = 1$, or
\begin{equation}\label{lazer2}
\int_0^T \!\!\! \int_{\{u>0\}} f_+(t,x)u(x) \,d x dt + \int_0^T \!\!\! \int_{\{u<0\}} f_-(t,x) u(x) \,d x dt < 0,
\end{equation}
for any $u\in N_\lambda $ with $\|u\| = 1$. Then the problem (\ref{A-eps-res}) admits a $T$-periodic mild solution.
\end{Th}

In the proof of preceding theorem, we use the following

\begin{Th}\label{th-reso}
Let $f \colon  [0,+\infty)\times\Omega\times \mathbb{R} \to \mathbb{R}$ satisfy the following condition:
\begin{equation}\label{lazer}
\int_0^T \!\!\! \int_{\{u>0\}} f_+(t,x)u(x) \,d x dt + \int_0^T \!\!\! \int_{\{u<0\}} f_-(t,x) u(x) \,d x dt
\neq 0
\end{equation}
for every $u\in N_\lambda $ with $\|u\| = 1$. Then there is $R > 0$ such that $\Psi_T(u)\neq u$ \ for $u\in X\setminus B(0,R)$, $g(u) \neq 0$ for $u\in N_\lambda  \setminus B(0,R) )$
and
\begin{equation}\label{ll3}
\mathrm{deg_{LS}}(I - \Psi_T, B(0,R)) = (-1)^{\mu(\lambda) + \dim N_\lambda }\,\mathrm{deg_B}(g, B(0,R)\cap N_\lambda )
\end{equation}
where $\mu(\lambda)$ is the sum of the algebraic multiplicities of the eigenvalues of $e^{\lambda T} S_A(T)\colon X\to X$ lying in $(1,+\infty)$.
\end{Th}
We shall use the following lemma
\begin{Lem}\label{lem-reso}
If $f \colon  [0,+\infty)\times\Omega\times \mathbb{R} \to \mathbb{R}$ satisfies (\ref{lazer}), then there is $R_0 > 0$ such that $g(u) \neq 0$ for $u\in N_\lambda $ with $\|u\| \ge R_0$.
\end{Lem}
\begin{proof}
Suppose the assertion is false. Then there is a sequence $(u_n)\subset N_\lambda $ such that $g(u_n) = 0$ for $n\ge 1$ and $\|u_n\| \to +\infty$ as $n\to +\infty$. Define $z_n := u_n/\|u_n\|$ for $n\ge 1$. Since $(z_n)\subset N_\lambda $ and $N_\lambda $ is a finite dimensional space, $(z_n)$ is relatively compact. We can assume that there is $z_0\in N_\lambda $ with $\|z_0\| = 1$ such that $z_n \to z_0$ as $n\to +\infty$. Additionally, we can suppose that $z_n(x) \to z_0(x)$ as $n\to +\infty$ for almost every $x\in\Omega$. Let
\begin{equation}\label{ozn}
\Omega_+ := \{x\in\Omega \ | \ z_0(x) > 0\} \qquad\text{and}\qquad
\Omega_- := \{x\in\Omega \ | \ z_0(x) < 0\}.
\end{equation}
Then, by (\ref{gr3}), we have
\begin{equation*}
0 = \langle g(u_n),z_0 \rangle = \int_0^T \!\!\! \int_\Omega f(t,x,u_n(x))z_0(x) \,d x dt,\qquad\text{for}\quad n\ge 1
\end{equation*}
and therefore
\begin{equation}\label{w2}
\int_0^T \!\!\! \int_{\Omega_+} f(t,x,z_n(x)\|u_n\|)z_0(x) \,d x dt
 + \int_0^T \!\!\! \int_{\Omega_-} f(t,x,z_n(x)\|u_n\|)z_0(x) \,d x dt = 0,
\end{equation}
for $n\ge 1$. Note that, for fixed $t\in [0,T]$, the convergence $f(t,x,z_n(x)\|u_n\|) \to f_+(t,x)$ by $n\to +\infty$ occurs for almost every $x\in\Omega_+$. Since the domain $\Omega$ has finite measure, $z_0\in L^2(\Omega)\subset L^1(\Omega)$. From the boundedness of $f$ and the dominated convergence theorem, we infer that, for any $t\in [0,T]$,
\begin{equation}\label{lll1}
\int_{\Omega_+} f(t,x,z_n(x)\|u_n\|) z_0(x)\,d x \to \int_{\Omega_+} f_+(t,x) z_0(x)\,d x \qquad\mathrm{as}\quad n\to +\infty.
\end{equation}
The function $\varphi_n^+\colon [0,T] \to \mathbb{R}$ given by
\begin{equation*}
\varphi_n^+(t):=\int_{\Omega_+} f(t,x,z_n(x)\|u_n\|) z_0(x)\,d x = \langle F(t,u_n),\max(z_0,0) \rangle \qquad\mathrm{for}\quad t\in[0,T]
\end{equation*}
is continuous and furthermore $|\varphi_n^+(t)| \le m\|z_0\|_{L^1(\Omega)} < +\infty$ for $t\in [0,T]$. By use of (\ref{lll1}) and the dominated convergence theorem, we deduce that
\begin{equation*}
\int_0^T \!\!\! \int_{\Omega_+} f(t,x,z_n(x)\|u_n\|) z_0(x)\,d x dt \to \int_0^T \!\!\! \int_{\Omega_+} f_+(t,x) z_0(x)\,d x dt
\end{equation*}
as $n\to +\infty$. Proceeding in the same way, we also find that
\begin{equation*}
\int_0^T \!\!\! \int_{\Omega_-} f(t,x,z_n(x)\|u_n\|) z_0(x)\,d x dt \to \int_0^T \!\!\! \int_{\Omega_-} f_-(t,x) z_0(x)\,d x dt
\end{equation*}
as $n\to +\infty$. In consequence, after passing to the limit in (\ref{w2})
\begin{equation*}
\int_0^T \!\!\! \int_{\Omega_+} f_+(t,x) z_0(x)\,d x dt + \int_0^T \!\!\! \int_{\Omega_-} f_-(t,x) z_0(x)\,d x dt = 0
\end{equation*}
for $z_0\in N_\lambda$ with $\|z_0\| = 1$, contrary to (\ref{lazer}), which completes the proof.
\end{proof}
\begin{proof}[Proof of Theorem \ref{th-reso}]
Consider the following differential problem
\begin{equation*}
\dot u(t) = - A u(t) + \lambda u(t) +\varepsilon F(t,u(t)),  \qquad t > 0
\end{equation*}
where $\varepsilon$ is a parameter from $[0,1]$ and let $\Upsilon_t\colon [0,1]\times X \to X$ be the translations along trajectories operator for this equation.  The previous lemma shows that there is $R_0 > 0$ such that $g(u) \neq 0$ for $u\in N_\lambda $ with $\|u\| \ge R_0$. We claim that there is $R_1 \ge R_0$ such that
\begin{equation}\label{row1aa}
\Upsilon_T(\varepsilon,u) \neq u \qquad\mathrm{for}\quad \varepsilon \in (0,1], \quad u\in X, \ \|u\| \ge R_1.
\end{equation}
Conversely, suppose that there are sequences $(\varepsilon_n)$ in $(0,1]$ and $(u_n)$ in $X$ such that
\begin{equation}\label{row1}
\Upsilon_T(\varepsilon_n,u_n) = u_n \qquad\mathrm{for}\quad n\ge 1
\end{equation}
and $\|u_n\| \to + \infty$ as $n\to + \infty$. For every $n\ge 1$, set
$w_n := w(\,\cdot \, ; \varepsilon_n, u_n)$ where $w(\,\cdot \, ; \varepsilon, u)$ is a mild solution of
$$\dot w(t) = - A w(t) + \lambda w(t) +\varepsilon F(t,w(t))$$
starting at $u$. Then
\begin{equation}\label{v1}
w_n(t) = e^{\lambda  t}S_A(t)u_n + \varepsilon_n  \int_0^t e^{\lambda  (t - s)}S_A(t - s)F(s,w_n(s)) \,d s
\end{equation}
for $n\ge1 $ and $t\in[0,+\infty)$. Putting $t := T$ in the above equation, by (\ref{row1}), we infer that
\begin{equation}\label{v3}
z_n = e^{\lambda T }S_A(T)z_n + v_n(T),
\end{equation}
with $z_n := u_n /\|u_n\|$ and
\begin{equation*}
v_n(t) := \frac{\varepsilon_n}{\|u_n\|} \int_0^t e^{\lambda  (t - s)}S_A(t - s)F(s,w_n(s)) \,d s \qquad\mathrm{for}\quad n\ge 1, \quad t\in[0,+\infty).
\end{equation*}
Observe that, for any $t\in [0,T]$ and $n\ge 1$, we have
\begin{equation}\label{v2}
\|v_n(t)\| \le \frac{1}{\|u_n\|}\int_0^t M e^{(\omega + \lambda)(t - s)}\|F(s,w_n(s))\| \,d s
\le m  \nu(\Omega)^{1/2}M e^{T(|\omega| + |\lambda|)}/\|u_n\|
\end{equation}
where the constants $M\ge 1$ and $\omega\in\mathbb{R}$ are such that $\|S_A(t)\|\le M e^{\omega t}$ for $t\ge 0$ and $\nu$ stands for the Lebesgue measure. Hence
\begin{equation}\label{v4}
v_n(t) \to 0 \qquad\mathrm{for}\quad t\in [0,T] \qquad\mathrm{as}\quad n\to +\infty,
\end{equation}
and, in particular, set $\{v_n(T)\}_{n\ge 1}$ is relatively compact. In view of (\ref{v3})
\begin{equation}\label{row2}
\{z_n\}_{n\ge 1} \subset e^{\lambda T }S_A(T)\left(\{z_n\}_{n\ge 1}\right) + \{v_n(T)\}_{n\ge 1},
\end{equation}
and therefore, by the compactness of $\{S_A(t)\}_{t\ge 0}$ we see that $\{z_n\}_{n\ge 1}$ has a convergent subsequence. Without loss of generality we may assume that
$z_n \to z_0$ as $n\to +\infty$ and $z_n(x) \to z_0(x)$ for almost every $x\in\Omega$, where $z_0\in X$ is such that $\|z_0\| = 1$. Passing to the limit in (\ref{v3}), as $n\to +\infty$, and using (\ref{v4}), we find that $z_0 = e^{\lambda T }S_A(T)z_0$, hence that $z_0 \in \mathrm{Ker}\,(I - e^{\lambda T}S_A(T))$ and finally, by condition (A3), that
\begin{equation}\label{rn}
z_0 \in \mathrm{Ker}\,(\lambda I - A) = \mathrm{Ker}\,(\lambda I - A^*).
\end{equation}
Thus Remark \ref{ker-a} (a)  leads to
\begin{equation}\label{v5}
z_0 \in \mathrm{Ker}\,(I - e^{\lambda  t}S_A(t)) \qquad\mathrm{for}\quad t\ge 0.
\end{equation}
From (\ref{v1}) we deduce that
\begin{equation*}
\frac{1}{\|u_n\|}(w_n(t) - u_n) = e^{\lambda  t}S_A(t)z_n - z_n + v_n(t) \qquad\mathrm{for}\quad t\in [0,T],
\end{equation*}
which by (\ref{v4}) and (\ref{v5}) gives
\begin{equation}\label{v9}
\frac{1}{\|u_n\|}(w_n(t) - u_n) \to 0 \qquad\mathrm{for}\quad  t\in [0,T] \quad \text{ as } n\to +\infty.
\end{equation}
If we again take $t := T$ in (\ref{v1}) and act with the scalar product operation $\langle\,\cdot \, , z_0\rangle$, we obtain
\begin{equation*}
\langle u_n,z_0 \rangle =  \langle e^{\lambda T}S_A(T)u_n, z_0 \rangle + \varepsilon_n  \int_0^T e^{\lambda (T - s)} \langle S_A(T - s)F(s,w_n(s)),z_0 \rangle \,d s.
\end{equation*}
Since $X$ is Hilbert space, by \cite[Corollary 1.10.6]{Pazy}, the family $\{S_A(t)^*\}_{t\ge0}$ of the adjoint operators is a $C_0$ semigroup on $X$ with the generator $-A^*$, i.e.
\begin{equation}\label{row4}
S_A(t)^* = S_{A^*}(t) \qquad\mathrm{for}\quad t\ge 0.
\end{equation}
Remark \ref{ker-a} (a) and (\ref{rn}) imply that $z_0\in\mathrm{Ker}\,(I - e^{\lambda t }S_{A^*}(t))$ for $t\ge 0$ and consequently, by (\ref{row4}), $z_0\in\mathrm{Ker}\,(I - e^{\lambda t }S_{A}(t)^*)$ for $t\ge 0$. Thus
\begin{align*}
\langle u_n,z_0 \rangle & = \langle u_n,e^{\lambda T} S_{A}(T)^*z_0 \rangle + \varepsilon_n  \int_0^T e^{\lambda  (T - s)} \langle F(s,w_n(s)), S_A(T - s)^* z_0 \rangle \,d s \\
& = \langle u_n,z_0 \rangle + \varepsilon_n  \int_0^T \langle F(s,w_n(s)), z_0 \rangle \,d s,
\end{align*}
and therefore
\begin{equation*}
\int_0^T \langle F(s,w_n(s)), z_0 \rangle \,d s = 0 \qquad\mathrm{for}\quad n\ge 1.
\end{equation*}
We have further
\begin{multline} \label{v14}
0 = \int_0^T \!\!\! \int_\Omega f(s,x,w_n(s)(x))z_0(x) \,d x ds \\
= \int_0^T \!\!\! \int_{\Omega_+} f(s,x,w_n(s)(x))z_0(x) \,d x ds
+ \int_0^T \!\!\! \int_{\Omega_-} f(s,x,w_n(s)(x))z_0(x) \,d x ds,
\end{multline}
where the sets $\Omega_+$ and $\Omega_-$ are given by (\ref{ozn}). Given $s\in[0,T]$, we claim that
\begin{equation}\label{v11}
\varphi_n^+(s):=\int_{\Omega_+} f(s,x,w_n(s)(x))z_0(x) \,d x \to
\int_{\Omega_+} f_+(s,x) z_0(x) \,d x
\end{equation}
and
\begin{equation}\label{v12}
\varphi_n^-(s):=\int_{\Omega_-} f(s,x,w_n(s)(x))z_0(x) \,d x \to
\int_{\Omega_-} f_-(s,x) z_0(x) \,d x
\end{equation}
as $n\to \infty$. Since the proofs of (\ref{v11}) and (\ref{v12}) are analogous, we consider only the former limit. We show that every sequence $(n_k)$ of natural numbers has a subsequence $(n_{k_l})$ such that
\begin{equation}\label{row5}
\int_{\Omega_+} f(s,x,(h_{n_{k_l}}(s,x) + z_{n_{k_l}}(x))\|u_{n_{k_l}}\|)z_0(x) \,d x \to
\int_{\Omega_+} f_+(s,x) z_0(x) \,d x
\end{equation}
as $n\to +\infty$ with
\[h_n(s,x) := (w_n(s)(x) - u_n(x))/\|u_n\| \qquad\mathrm{for}\quad x\in\Omega, \quad n\ge 1.\]
Due to (\ref{v9}), one can choose a subsequence $(h_{n_{k_l}}(s,\,\cdot\,))$ of $(h_{n_k}(s, \, \cdot \,))$ such that $h_{n_{k_l}}(s,x) \to 0$ for almost every $x\in\Omega$. Hence
\begin{equation}
h_{n_{k_l}}(s,x) + z_{n_{k_l}}(x) \to z_0(x) > 0 \qquad\mathrm{as}\quad n\to +\infty
\end{equation}
for almost every $x\in \Omega_+$ and consequently
\begin{equation}
f(s,x,(h_{n_{k_l}}(s,x) + z_{n_{k_l}}(x))\|u_{n_{k_l}}\|) \to  f_+(s,x) \qquad\mathrm{as}\quad n\to +\infty
\end{equation}
for almost every $x\in \Omega_+$. Since $z_0\in L^2(\Omega)\subset L^1(\Omega)$ and $f$ is bounded, from the Lebesgue dominated convergence theorem, we have the convergence (\ref{row5}) and hence (\ref{v11}). Further, for any $s\in [0,T]$ and $n\ge 1$, one has
\begin{align}\label{v-bou}
|\varphi_n^+(s)| & \le \int_{\Omega_+} |f(s,x,w_n(s)(x))z_0(x)| \,d x
\le m\int_{\Omega_+} |z_0(x)| \,d x \le m \|z_0\|_{L^1(\Omega)}.
\end{align}
and similarly
\begin{equation}\label{v-bou2}
|\varphi_n^-(s)| \le m \|z_0\|_{L^1(\Omega)}\qquad\text{for}\quad t\in[0,T]\quad\text{and}\quad n\ge 1.
\end{equation}
Since
\begin{equation*}
\varphi_n^+(s)= \langle F(s,w_n(s)), \max(z_0,0) \rangle \quad\text{ and }\quad \varphi_n^-(s)= \langle F(s,w_n(s)), \min(z_0,0) \rangle
\end{equation*}
for $s\in[0,T]$ and $n\ge 1$, functions $\varphi_n^+$ and $\varphi_n^-$ are continuous on $[0,T]$. Using (\ref{v11}), (\ref{v12}), (\ref{v-bou}), (\ref{v-bou2}) and the dominated convergence theorem, after passing to the limit in (\ref{v14}), we infer that
\begin{equation}
\int_0^T \!\!\! \int_{\Omega_+} f_+(s,x) z_0(x) \,d x ds + \int_0^T \!\!\! \int_{\Omega_-} f_-(s,x) z_0(x) \,d x ds = 0,
\end{equation}
which contradicts (\ref{lazer}), since $z_0\in N_\lambda$ and $\|z_0\| = 1$ and, in consequence, proves (\ref{row1aa}). \\
\indent Let $R:=R_1$. By the homotopy invariance of topological degree, for any $\varepsilon\in (0,1]$, we have
\begin{align}\label{w7}
\mathrm{deg_{LS}}(I - \Psi_T, B(0,R)) & =\mathrm{deg_{LS}}(I - \Upsilon_T(1, \,\cdot \, ), B(0,R)) \\ & = \mathrm{deg_{LS}}(I - \Upsilon_T(\varepsilon, \, \cdot \, ), B(0,R)) \notag.
\end{align}
Since $A$ has compact resolvents $\mathrm{Ker}\,(A^* - \lambda I )^\bot = \mathrm{Im}\,(A - \lambda I )$ and therefore, by (A3), $X$ admits the direct sum decomposition $$X = N_\lambda\oplus \mathrm{Im}\,(A - \lambda I ).$$ Clearly the range and kernel of $A$ are invariant under $S_A(t)$ for $t\ge 0$, hence putting $M:=\mathrm{Im}\,(\lambda I - A)$, condition (A2) is satisfied for $A - \lambda I$. Moreover $R \ge R_0$ and therefore, we also have that $g(u) \neq 0$ for $u\in N_\lambda $ with $\|u\| \ge R$. Let $U := B(0,R)\cap N_\lambda $ and $V := B(0,R) \cap M$. Then $g(u) \neq 0$ for $u\in \partial_{N_\lambda} U$ and clearly
\begin{equation}\label{f9}
B(0,R) \subset U\oplus V.
\end{equation}
Therefore, by Theorem \ref{th-aver-ker}, there is $\varepsilon_0 \in (0,1)$ such that, for any $\varepsilon \in (0,\varepsilon_0]$ and $u\in \partial(U \oplus V)$, $\Upsilon_T(\varepsilon,u)\neq u$ and
\begin{equation}\label{w9}
\mathrm{deg_{LS}}(I - \Upsilon_T(\varepsilon,\,\cdot\,), U \oplus V) = (-1)^{\mu(\lambda) + \dim N_\lambda }\deg_B(g, U),
\end{equation}
where $\mu(\lambda)$ is the sum of algebraic multiplicities of eigenvalues of $S_{A - \lambda I}(T)$ in $(1,+\infty)$. In view of (\ref{f9}) and the fact that $R = R_1$ satisfies (\ref{row1aa}), we infer that \[\{u\in U\oplus V \ | \ \Upsilon_T(\varepsilon_0,u) = u\} \subset B(0,R)\]
and, by the excision property,
\begin{equation}\label{w11}
\mathrm{deg_{LS}}(I - \Upsilon_T(\varepsilon_0,\,\cdot\,), U \oplus V)
= \mathrm{deg_{LS}}(I - \Upsilon_T(\varepsilon_0,\,\cdot\,), B(0,R)).
\end{equation}
Combining (\ref{w9}) with (\ref{w11}) yields
\begin{equation}\label{w10}
\mathrm{deg_{LS}}(I - \Upsilon_T(\varepsilon_0,\,\cdot\,), B(0,R)) = (-1)^{\mu(\lambda) + \dim N_\lambda }\,\mathrm{deg_B}(g, U),
\end{equation}
which together with (\ref{w7}) implies
\begin{equation}\label{w12}
\mathrm{deg_{LS}}(I - \Psi_T, B(0,R)) = (-1)^{\mu(\lambda) + \dim N_\lambda }\,\mathrm{deg_B}(g, U)
\end{equation}
and the proof is complete.
\end{proof}

The following proposition allows us to determine the Brouwer degree of the mapping $g$.

\begin{Prop}\label{th-guid-fun}
\ \\[2mm]
\makebox[10mm][r]{(i)} \parbox[t]{138mm}{If condition (\ref{lazer1}) holds then there is $R_0 > 0$ such that $g(u)\neq 0$ for $u\in N_\lambda $ with $\|u\| \ge R_0$ and \\[-3mm]
\begin{equation*}
\mathrm{deg_B}(g, B(0,R)) = 1 \qquad\mathrm{for}\quad R\ge R_0.
\end{equation*}} \\
\makebox[10mm][r]{(ii)} \parbox[t]{138mm}{If condition (\ref{lazer2}) holds then  there is $R_0 > 0$ such that $g(u)\neq 0$ for $u\in N_\lambda $ with $\|u\| \ge R_0$ and \\[-3mm]
\begin{equation*}
\mathrm{deg_B}(g, B(0,R)) = (-1)^{\dim N_\lambda } \qquad\mathrm{for}\quad R\ge R_0.
\end{equation*}}
\end{Prop}
\begin{proof}
(i) We begin by proving that there exists $R_0 > 0$ such that
\begin{equation}\label{gf1}
\langle g(u),u \rangle > 0 \qquad\mathrm{for}\quad u\in N_\lambda , \ \|u\| \ge R_0.
\end{equation}
Arguing by contradiction, suppose that there is a sequence $(u_n)\subset N_\lambda $ such that $\|u_n\| \to +\infty$ as $n\to +\infty$ and $\langle g(u_n),u_n \rangle \le 0$, for $n \ge 1$. For every $n\ge 1$, write $z_n := u_n/\|u_n\|$. Since $(z_n)$ is bounded and contained in the finite dimensional space $N_\lambda $, it contains a convergent subsequence. Without loss of generality we may assume that there is $z_0\in N_\lambda $ with $\|z_0\| = 1$ such that $z_n \to z_0$ as $n\to +\infty$ and $z_n(x) \to z_0(x)$ as $n\to +\infty$ for almost every $x\in\Omega$. Recalling the notational convention (\ref{ozn}), we have
\begin{align}\label{gr4}
0 \ge \langle g(u_n),z_n \rangle & = \langle g(u_n),z_n - z_0 \rangle + \langle g(u_n),z_0 \rangle \\ \nonumber
& =  \int_0^T \!\!\! \int_\Omega f(t,x,u_n(x))z_0(x) \,d x dt + \langle g(u_n),z_n - z_0 \rangle \\ \nonumber
& =  \int_0^T \!\!\!\int_{\Omega_+} f(t,x,z_n(x)\|u_n\|)z_0(x) \,d x dt \\ \nonumber
& \quad + \int_0^T \!\!\! \int_{\Omega_-} f(t,x,z_n(x)\|u_n\|)z_0(x) \,d x dt + \langle g(u_n),z_n - z_0 \rangle.
\end{align}
On the other hand, if we fix $t\in [0,T]$, then, by the condition (d), we have
\begin{equation}\label{l1}
f(t,x,z_n(x)\|u_n\|) \to f_+(t,x) \qquad\mathrm{as}\quad n\to +\infty
\end{equation}
for almost every $x\in \Omega_+$. Since $f$ is assumed to be bounded and $z_0\in L^1(\Omega)$, by the dominated convergence theorem, (\ref{l1}) shows that
\begin{equation}\label{lll1a}
\int_{\Omega_+} f(t,x,z_n(x)\|u_n\|) z_0(x)\,d x \to
\int_{\Omega_+} f_+(t,x) z_0(x)\,d x
\end{equation}
as $n\to \infty$. Let $\varphi_n^+\colon [0,T] \to \mathbb{R}$ be
given by
\begin{equation*}
\varphi_n^+(t):=\int_{\Omega_+} f(t,x,z_n(x)\|u_n\|) z_0(x)\,d x = \langle F(t,u_n),\max(z_0,0) \rangle
\end{equation*}
for $t\in [0,T]$. The function $\varphi_n^+$ is evidently continuous and $|\varphi_n^+(t)| \le m\|z_0\|_{L^1(\Omega)}$ for $t\in [0,T]$. Applying (\ref{lll1a}) and the dominated convergence theorem, we find that
\begin{equation}\label{gg1}
\int_0^T \!\!\! \int_{\Omega_+} f(t,x,z_n(x)\|u_n\|)z_0(x) \,d x dt
\to \int_0^T \!\!\!\int_{\Omega_+} f_+(t,x) \,d x dt,
\end{equation}
as $n\to +\infty$. Proceeding in the same way, we infer that
\begin{equation}\label{gg2}
\int_0^T \!\!\! \int_{\Omega_-} f(t,x,z_n(x)\|u_n\|)z_0(x) \,d x dt
\to \int_0^T \!\!\!\int_{\Omega_-} f_-(t,x) \,d x dt,
\end{equation}
as $n\to +\infty$. Since the sequence $(g(u_n))$ is bounded, we see that
\begin{equation}\label{ja19}
|\langle g(u_n),z_n - z_0 \rangle | \le \|g(u_n)\|\|z_n - z_0\| \to 0 \qquad\mathrm{as}\quad n\to +\infty.
\end{equation}
By (\ref{gg1}), (\ref{gg2}), (\ref{ja19}), letting $n\to +\infty$ in (\ref{gr4}), we assert that
\begin{equation*}
\int_0^T \!\!\! \int_{\Omega_+} f_+(t,x) z_0(x) \,d x dt
+ \int_0^T \!\!\!\int_{\Omega_-} f_-(t,x) z_0(x) \,d x dt \le 0,
\end{equation*}
contrary to (\ref{lazer1}). \\
\indent Now, for any $R > R_0$, the mapping $H\colon [0,1]\times N_\lambda
\to N_\lambda $ given by \[H(s,u) := s g(u) + (1 - s)u \qquad\text{for}\quad
u\in N_\lambda ,\] has no zeros for $s\in [0,1]$ and $u\in N_\lambda $ with $\|u\| =
R$. If it were not true, then there would be $H(s,u) = 0$, for
some $s\in [0,1]$ and $u\in N_\lambda $ with $\|u\| = R$, and in
consequence,
\begin{align*}
0 = \langle H(s,u),u \rangle = s \langle g(u),u \rangle + (1 - s) \langle u,u \rangle.
\end{align*}
If $s = 0$ then $0=\|u\|^2 = R^2$, which is impossible. If $s\in (0,1]$, then
$0 \ge \langle g(u),u \rangle$, which contradicts (\ref{gf1}). By the homotopy invariance of the topological degree
\begin{align*}
\mathrm{deg_B}(g, B(0,R)) & = \mathrm{deg_B}(H(1,\,\cdot\,), B(0,R)) = \mathrm{deg_B}(H(0,\,\cdot\,), B(0,R)) \\
& = \mathrm{deg_B}(I, B(0,R)) = 1,
\end{align*}
and the proof of (i) is complete. \\
(ii) Proceeding by analogy to (i), we obtain the existence of $R_0 > 0$ such that
\begin{equation}\label{gf11}
\langle g(u),u \rangle < 0 \qquad\mathrm{for}\quad \|u\| \ge R_0.
\end{equation}
This implies, that for every $R > R_0$, the homotopy $H\colon
[0,1]\times N_\lambda  \to N_\lambda $ given by
\[H(s,u) := s g(u) - (1 - s)u \qquad\text{for}\quad
u\in N_\lambda \] is such that $H(s,u) \neq 0$ for $s\in [0,1]$ and $u\in
N_\lambda $ with $\|u\| = R$. Indeed, if $H(s,u) = 0$ for some $s\in [0,1]$
and $u\in N_\lambda $ with $\|u\| = R$, then
\begin{align*}
0 = \langle H(s,u),u \rangle = s \langle g(u),u \rangle - (1 - s)\langle u,u \rangle.
\end{align*}
Hence, if $s\in (0,1]$, then $\langle g(u),u \rangle \ge 0$, contrary to (\ref{gf11}). If $s=0$, then
$R^2 = \|u\|^2 = 0$, and again a contradiction. In consequence, by the homotopy invariance,
\begin{equation*}
\mathrm{deg_B}(g, B(0,R)) = \mathrm{deg_B}(-I, B(0,R)) = (-1)^{\dim N_\lambda },
\end{equation*}
as desired.
\end{proof}
\begin{proof}[Proof Theorem \ref{th-guid-fun1} ]
Theorem \ref{th-reso} asserts that there is $R>0$ such that $\Psi_T(u)\neq u$ for $u\in X\setminus B(0,R)$, $g(u) \neq 0$ for $u\in N_\lambda \setminus B(0,R)$ and
\begin{equation}\label{ll3b}
\mathrm{deg_{LS}}(I - \Psi_T, B(0,R)) = (-1)^{\mu(\lambda) + \dim N_\lambda }\mathrm{deg_B}(g, B(0,R) \cap N_\lambda).
\end{equation}
In view of Proposition \ref{th-guid-fun}, we obtain the existence of $R_0 > R$ such that either $\deg(g, B(0,R_0) \cap N_\lambda ) = 1$, when (\ref{lazer1}) is satisfied, or
$\deg(g, B(0,R_0) \cap N_\lambda ) = (-1)^{\dim N_\lambda }$, in the case of condition (\ref{lazer2}).
By the inclusion $\{u\in B(0,R_0)\cap N_\lambda  \ | \ g(u) = 0\}\subset B(0,R)\cap N_\lambda $, we see that $$\deg(g, B(0,R) \cap N_\lambda ) = \deg(g, B(0,R_0) \cap N_\lambda ) = \pm 1$$ and, by (\ref{ll3b}),
\begin{align*}
\mathrm{deg_{LS}}(I - \Psi_T, B(0,R)) = (-1)^{\mu(\lambda) + \dim N_\lambda }\mathrm{deg_B}(g, B(0,R)\cap N_\lambda ) = \pm 1.
\end{align*}
Thus, by the existence property, we find that there is a fixed point of $\Psi_T$ and in consequence a $T$-periodic mild solution of (\ref{A-eps-res}).
\end{proof}

In the particular case when the linear operator $A$ is self-adjoint and $-A$ is a generator of a compact $C_0$ semigroup $\{S_A(t)\}_{t\ge 0}$ of bounded linear operators on $X$, the spectrum $\sigma(A)$ is real and consists of eigenvalues $\lambda_1 < \lambda_2 < \lambda_3 < \ldots < \lambda_k < \ldots$ which form a sequence convergent to infinity. By Proposition \ref{th-sem-spec}, for every $t > 0$, $\{e^{- \lambda_k t}\}_{k\ge 1}$ is the sequence of nonzero eigenvalues of $S_A(t)$ and
\begin{equation}\label{re-ker}
\mathrm{Ker}\,(\lambda_k I - A) = \mathrm{Ker}\,(e^{- \lambda_k t}I - S_A(t)) \qquad\mathrm{for}\quad k \ge 1.
\end{equation}
In consequence, we see that (A3) holds.

\begin{Cor}\label{th-reso-adjoint}
Let $A$ be a self-adjoint operator such that $-A$ is a generator of a compact $C_0$ semigroup $\{S_A(t)\}_{t\ge 0}$ and let $f\colon [0,+\infty)\times\Omega\times \mathbb{R} \to \mathbb{R}$ satisfy the Landesman--Lazer type condition (\ref{lazer}). If $\lambda = \lambda_k$ for some $k\ge 1$, then there is $R>0$ such that $\Psi_T(u)\neq u$ for $u\in X\setminus B(0,R)$, $g(u) \neq 0$ for
$u\in N_{\lambda_k} \setminus B(0,R)$ and
\begin{equation}\label{rowad}
\mathrm{deg_{LS}}(I - \Psi_T, B(0,R)) = (-1)^{d_k}\mathrm{deg_B}(g, B(0,R)\cap N_{\lambda_k}),
\end{equation}
where $d_k := \sum_{i=1}^{k-1}\dim \mathrm{Ker}\,(\lambda_i I - A)$ for $k \ge 1$. In particular, if either condition (\ref{lazer1}) or (\ref{lazer2}) is satisfied then (\ref{A-eps-res}) has mild solution.
\end{Cor}
\begin{proof}
To see (\ref{rowad}), it is enough to check that $d_k = \mu(\lambda_k) + \dim N_{\lambda_k}$ for $k\ge 1$. Since
\[e^{(\lambda_k-\lambda_1)T} > e^{(\lambda_k - \lambda_2)T} > \ldots > e^{(\lambda_k - \lambda_{k - 1})T} \]
are eigenvalues of $e^{\lambda_k T}S_A(T)$ which are greater than $1$, for $k = 1$ it is evident that $\mu(\lambda_k) = 0$ and $d_1 = \mu(\lambda_k) + \dim N_{\lambda_k}$. The operator $S_A(T)$ is also self-adjoint and therefore the geometric and the algebraic multiplicity of each eigenvalue coincide. Hence
\begin{equation}\label{war-wl}
\mu(\lambda_k) = \sum_{i=1}^{k-1}\dim\mathrm{Ker}\,(e^{-\lambda_i T}I - S_A(T))\qquad\text{for}\quad k\ge 2.
\end{equation}
From (\ref{re-ker}) and (\ref{war-wl}), we deduce that
\begin{equation*}
\mu(\lambda_k) = \sum_{i=1}^{k-1}\dim \mathrm{Ker}\,(\lambda_i I - A) = d_k - \dim N_{\lambda_k}
\end{equation*}
and finally that $d_k = \mu(\lambda_k) + \dim N_{\lambda_k}$ for every $k\ge 1$, as desired. The formula (\ref{rowad}) together with Proposition \ref{th-guid-fun} leads to existence of mild solution of (\ref{A-eps-res}) provided either condition (\ref{lazer1}) or (\ref{lazer2}) is satisfied.
\end{proof}

\section{Applications}

Let $\Omega\subset\mathbb{R}^n$, $n\ge 1$, be an open bounded connected set with $C^1$ boundary. We recall that $\|\cdot\|$ and $\langle\,\cdot \, , \,\cdot\,\rangle$ denote, similarly as before, the norm and the scalar product on $X = L^2(\Omega)$, respectively. For $u\in H^1(\Omega)$, we will denote by $D_k u$, the $k$-th weak derivative of $u$.
\subsubsection*{Laplacian with the Neumann boundary conditions}
We begin with the $T$-periodic parabolic problem
\begin{equation}\label{del-f-lam}
\left\{ \begin{array}{ll}
\displaystyle \frac{\partial u}{\partial t} = \Delta u + \varepsilon f(t, x, u) & \qquad
\mathrm{in}\quad (0,+\infty)\times\Omega \\[3mm]
\displaystyle \frac{\partial u}{\partial n}(t,x) = 0 & \qquad
\mathrm{on}\quad [0,+\infty)\times\partial\Omega \\[3mm]
u(t,x) = u(t + T,x) & \qquad
\mathrm{in} \quad [0,+\infty)\times\Omega,
\end{array} \right.
\end{equation}
where $\varepsilon \in [0,1]$ is a parameter and $f\colon [0,+\infty)\times\Omega\times\mathbb{R} \to \mathbb{R}$ is a continuous mapping which is required to satisfy conditions (a), (b) and (c) from the previous section. We put (\ref{del-f-lam}) into an abstract setting. To this end let $A\colon D(A) \to X$ be a linear operator such that $-A$ is the Laplacian with the Neumann boundary conditions, i.e.
\begin{align*}
D(A) & := \bigg\{u\in H^1(\Omega) \ | \ \text{ there is }g\in L^2(\Omega) \text{ such that } \\
& \hspace{40mm}\left. \int_\Omega \nabla u \nabla h \ dx = \int_\Omega g h \ dx \ \text{ for } \  h\in H^1(\Omega)\right\}, \\
A u & := g, \text{ where } g \text{ is as above},
\end{align*}
and define $F\colon [0,+\infty)\times X \to X$ to be a mapping given by the formula
\begin{equation}\label{row3}
F(t,u)(x) := f(t,x,u(x)) \qquad\mathrm{for}\quad t\in [0,+\infty),\quad x\in\Omega.
\end{equation}
Then by the assumptions (a) and (b), it is well defined, continuous, bounded and Lipschitz uniformly with respect to time. Problem (\ref{del-f-lam}) may be considered in the abstract form
\begin{equation}\label{A-F-lam2}
\left\{ \begin{array}{ll}
\dot u(t) = - A u(t) +\varepsilon F(t,u(t)), & \qquad t > 0\\[1mm]
u(t)=u(t + T) &  \qquad t\ge 0\end{array} \right.
\end{equation}
where $\varepsilon\in [0,1]$ is a parameter. Solutions of (\ref{del-f-lam}) will be understand as mild solutions of (\ref{A-F-lam2}).
\begin{Th}
Let $g_0\colon \mathbb{R} \to \mathbb{R}$ be given by
\[g_0(y):=\int_0^T\!\!\!\int_\Omega f(t, x,  y) \,d x d t \qquad\mathrm{for}\quad y\in\mathbb{R}.\]
If real numbers $a$ and $b$ are such that $a < b$ and \ $g_0(a)\cdot g_0(b) < 0$, then there is $\varepsilon_0 > 0$ such that for $\varepsilon\in(0,\varepsilon_0]$, the problem (\ref{del-f-lam}) admits a solution.
\end{Th}
\begin{proof}
Since the spectrum of $A$ is real, condition (A1) is satisfied as a consequence of Remark \ref{ker-a}. It is known that $-A$ generates a compact $C_0$ semigroup on $X$ and $N := \mathrm{Ker}\, A$ is a one dimensional space. Furthermore, if we take $M:=\mathrm{Im}\, A$, then $M = N^\bot$ and hence $A$ satisfies also condition (A2). Let $P\colon X\to X$ be the orthogonal projection onto $N$ given by
\[P(u) := \frac{1}{\nu(\Omega)}(u,e)\cdot e \qquad\mathrm{for}\quad u\in X\] where $e\in L^2(\Omega)$ represents the constant equal to $1$ function and $\nu$ stands for the Lebesgue measure. Set
$U :=\{s\cdot e \ | \ s\in (a,b)\}$, $V :=\{u\in N^{\bot} \ | \ \|u\| < 1\}$ and let $g\colon N \to N$ be defined by $$g(u):=\int_0^T PF(t, u) \,d t \qquad\mathrm{for}\quad u\in N.$$ Then
$$g_0(y) = \nu(\Omega) \cdot K^{-1}(g(K(y)))\qquad\mathrm{for}\quad y\in\mathbb{R},$$
where $K\colon \mathbb{R} \to N$ is the linear homeomorphism given by $K(y) := y\cdot e$.
Since $g_0(a)\cdot g_0(b) < 0$, we have $\mathrm{deg_B}(g,U) = \mathrm{deg_B}(g_0, (a,b)) \neq 0$ and hence, by Corollary \ref{cor-aver}, there is $\varepsilon_0 \in (0,1)$ such that, for $\varepsilon\in (0,\varepsilon_0]$, problem (\ref{del-f-lam}) admits a solution as desired.
\end{proof}

\subsubsection*{Differential operator with the Dirichlet boundary conditions}

Suppose that $a^{ij} = a^{ji}\in C^1(\overline{\Omega})$ for $1\le i,j\le n$ and let $\theta > 0$ be such that $$a^{ij}(x)\xi_i \xi_j \ge \theta |\xi|^2 \qquad\text{for}\quad \xi = (\xi_1,\xi_2, \ldots, \xi_n)\in \mathbb{R}^n, \quad x\in\Omega.$$ We assume that $A \colon D(A) \to X$ is a linear operator given by the formula
\begin{align*}
D(A) & := \bigg\{u\in H^1_0(\Omega) \ | \ \text{ there is }g\in L^2(\Omega) \text{ such that } \\
& \hspace{30mm}\left. \int_\Omega a^{ij}(x)D_iu D_j h \ dx = \int_\Omega g h \ dx \ \text{ for } \  h\in H^1_0(\Omega)\right\}, \\
A u & := g, \text{ where } g \text{ is as above}.
\end{align*}
It is well known that $-A$ is self-adjoint and generates a compact $C_0$ semigroup on $X=L^2(\Omega)$. Let $\lambda_1 < \lambda_2 < \ldots < \lambda_k < \ldots $ be the sequence of distinct eigenvalues of $A$. We are concerned with a periodic parabolic problem of the form
\begin{equation}\label{A-elip}
\left\{ \begin{array}{ll}
u_t = D_i (a^{ij}D_ju) + \lambda_k u + f(t, x, u) & \qquad\text{in}\quad  (0,+\infty)\times\Omega \\
u(t,x) = 0 & \qquad\text{on}\quad [0,+\infty)\times\partial\Omega \\
u(t,x) = u(t + T,x) & \qquad\text{in}\quad [0,+\infty)\times\Omega,
\end{array} \right.
\end{equation}
where $\lambda_k$ is $k$-th eigenvalue of $A$ and $f\colon [0,+\infty)\times\Omega\times\mathbb{R}\to \mathbb{R}$ is as above. We write problem (\ref{A-elip}) in the abstract form
\begin{equation*}
\left\{ \begin{array}{ll}
\dot u(t) = - A u(t) + \lambda_k u(t) +  F(t,u(t)), & \qquad t > 0\\
u(t) = u(t + T) & \qquad t\ge 0 \end{array} \right.
\end{equation*}
where $F:[0,+\infty)\times X \to X$ is given by the formula (\ref{row3}). An immediate consequence of Corollary \ref{th-reso-adjoint} is the following
\begin{Th}
Suppose that $f \colon  [0,+\infty)\times\Omega\times\mathbb{R} \to \mathbb{R}$ is such that:
\begin{equation*}
\int_0^T \!\!\!\int_{\{u>0\}} f_+(t,x)u(x) \,d x dt + \int_0^T \!\!\!\int_{\{u<0\}} f_-(t,x) u(x) \,d x dt > 0,
\end{equation*}
for any $u\in \mathrm{Ker} \, A$ with $\|u\| = 1$, or
\begin{equation*}
\int_0^T \!\!\! \int_{\{u>0\}} f_+(t,x)u(x) \,d x dt + \int_0^T \!\!\! \int_{\{u<0\}} f_-(t,x) u(x) \,d x dt < 0,
\end{equation*}
for any $u\in \mathrm{Ker} \, A$ with $\|u\| = 1$. Then the problem (\ref{A-elip}) admits a $T$-periodic mild solution.
\end{Th}

\noindent {\footnotesize {\em Acknowledgements.} The author wishes to thank Prof. W. Kryszewski and Dr. A. \'Cwiszewski for helpful comments and suggestions, which raised the quality of this work.}

\end{document}